\documentclass[11pt, reqno]{article}

\usepackage{fullpage}
\usepackage{enumerate}
\usepackage{setspace}

\usepackage{ stmaryrd }
\usepackage[shortlabels]{enumitem}
\usepackage{amsmath,amsfonts,amssymb,graphics,amsthm, comment}
\usepackage{hyperref}
\hypersetup{
    colorlinks=true,
    linkcolor=blue,
    citecolor=red,
    urlcolor=blue,
    pdfborder={0 0 0}
}    

\usepackage[normalem]{ulem}
\usepackage{appendix}        

\usepackage{graphicx}
\usepackage{xcolor}
\usepackage{bbm}
\usepackage{wrapfig} 
\usepackage{xcolor}

\usepackage[font=sf, labelfont={sf,bf}, margin=1cm]{caption}

\usepackage{cleveref}
  \crefname{theorem}{Theorem}{Theorems}
  \crefname{thm}{Theorem}{Theorems}
  \crefname{lemma}{Lemma}{Lemmata}
  \crefname{lem}{Lemma}{Lemmata}
  \crefname{remark}{Remark}{Remarks}
  \crefname{prop}{Proposition}{Propositions}
  \crefname{proposition}{Proposition}{Propositions}
\crefname{notation}{Notation}{Notations}
\crefname{claim}{Claim}{Claims}
  \crefname{defn}{Definition}{Definitions}
  \crefname{corollary}{Corollary}{Corollaries}
  \crefname{section}{Section}{Sections}
  \crefname{figure}{Figure}{Figures}
    \crefname{assumption}{Assumption}{Assumptions}

\newtheorem{thm}{Theorem}[section]

\newtheorem{lemma}[thm]{Lemma}
\newtheorem{corollary}[thm]{Corollary}
\newtheorem{prop}[thm]{Proposition}
\newtheorem{proposition}[thm]{Proposition}

\newtheorem{question}[thm]{Question}

\numberwithin{equation}{section}

\theoremstyle{definition}
\newtheorem{remark}[thm]{Remark}

\def\cV{\mathcal{V}}

\def\cH{\mathcal{H}}
\def\cG{\mathcal{G}}
\def\cF{\mathcal{F}}
\def\cE{\mathcal{E}}
\def\cD{\mathcal{D}}
\def\cC{\mathcal{C}}
\def\cB{\mathcal{B}}
\def\cA{\mathcal{A}}

\newcommand{\bbZ}{\mathbb{Z}}

\def \ve {\varepsilon}

\def\R{\mathbb{R}}
\def\Z{\mathbb{Z}}
\def\N{\mathbb{N}}

\def\R{\mathbb{R}}

\def  \p- {p\textunderscore}

\def\eps{\varepsilon}
\def\ep{\varepsilon}

\usepackage{extarrows}

\title{Logarithmic variance for the height function of square-ice}
\author{Hugo Duminil--Copin\thanks{IHES, Universit\'e de Gen\`eve} \and Matan Harel\thanks{University of Tel Aviv} \and Benoit Laslier\thanks{Universit\'e de Paris--Diderot} \and Aran Raoufi \thanks{ETH} \and Gourab Ray\thanks{University of Victoria}}

\begin{document}
\maketitle
\begin{abstract}
In this article, we prove that the height function associated with the square-ice model (i.e.~the six-vertex model with $a=b=c=1$ on the square lattice), or, equivalently, of the uniform random homomorphisms from $\Z^2$ to $\Z$, has logarithmic variance. This establishes a strong form of roughness of this height function. 
\end{abstract}
\section{Introduction}

\subsection{Main results}

Two-dimensional models for random surfaces are one of the main subjects of interest of modern statistical physics. These models often undergo a phase transition between a {\em localized} phase where the random surface does not fluctuate (or equivalently, the variance of the height function at a point remains bounded), and a {\em delocalized} phase where it does, in the sense that the variance goes to infinity as the domain grows. In the latter, the model is usually predicted to have a Gaussian behaviour and to converge in the sense of distributions in the scaling limit to the Gaussian Free Field (GFF).

There are many models of random surfaces but only a few for which it is known whether the model is in its localized or delocalized phase. Even in cases where the random surface was proved to be delocalized, the convergence to GFF is far from understood {in the majority of cases}. The situation is particularly catastrophic in models where the surface is modeled as a function $h$ from the vertices of a graph $G$ to the integers such that $|h_v-h_u| =1$. We call such functions \emph{homomorphisms} or {\em height functions}. 

\begin{figure}[htb]
	\begin{center}
		\includegraphics[width=0.6\textwidth, page=1]{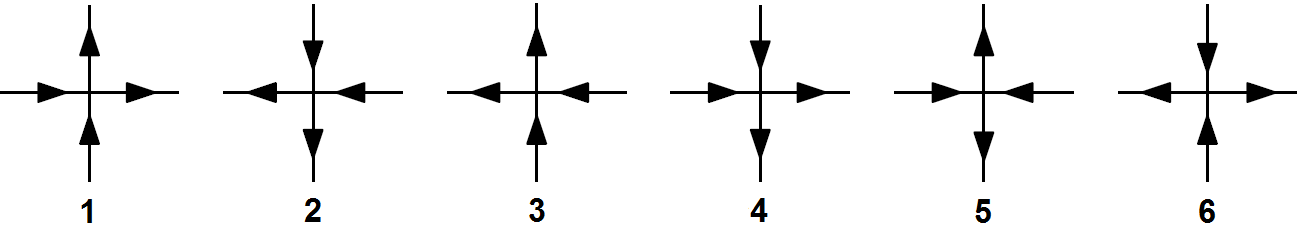}
	\end{center}
	\caption{The $6$ possibilities for vertices in the six-vertex model.
	Each possibility comes with a weight $a$, $b$ or $c$.}
	\label{fig:the_six_vertices}
\end{figure}

The six-vertex model was initially proposed by Pauling in 1935 in order to study the thermodynamic properties of ice, and in this paper we study a special case of this model.   Fix an integer $n$ and consider the torus $\mathbb T_n:=(\mathbb Z/n\mathbb Z)^2$ and its dual graph $\mathbb T_n^*$. Let $\omega$ be an arrow configuration on the edges of $\mathbb T_n^*$ assigning one of two orientations to each edge of the graph.  The six-vertex model is given by restricting $\omega$ to configurations that have an equal number of arrows entering and exiting each vertex of $\mathbb T_n^*$ -- a relation we call the {\em ice rule}. The rule leaves six possible configurations at each such vertex, depicted in Figure~\ref{fig:the_six_vertices}. 
Assign the weight $a$ to configurations 1 and 2, $b$ to 3 and 4, and $c$ to 5 and 6. 
The six-vertex model with weights $a,b,c$ consists in picking such a configuration at random with probability proportional to $a^{n_1 + n_2}b^{n_3 + n_4} c^{n_5 + n_6}$,
where $n_i$ is the number of vertices with configuration $i$ in $\omega$, if $\omega$ satisfies the ice-rule, and zero otherwise. 

Thanks to the ice-rule, six-vertex configurations are naturally associated with a {\em height function} $h$ on $\mathbb T_n$ defined by the property that the increment between the endpoints $u$ and $v$ of the edge $e$ is $+1$ if the associated arrow of the dual edge $e^*$ is crossed from left to right when going from $u$ to $v$. The height function is technically defined on $\mathbb{Z}^2$, which is the cover of $\mathbb{T}_n$, as a lift of the gradient, which provides a natural definition for $h_u - h_v$ for any $u,v \in \mathbb{T}_n$. On the subset of arrow configurations with as many up arrows as down arrows on each line, and as many left arrows as right arrows on each row, we obtain a well defined height function on $\mathbb T_n$. Note that $h$ is defined up to constant, and we will therefore often consider equivalence classes of $h$ for the relation $\sim$, where $h \sim h'$ if and only if $h-h'$ is constant. Also note that the height function partitions the lattice $\Z^2$ into vertices which always take odd values and vertices which always take even values. We call them respectively \emph{odd vertices} and \emph{even vertices}. Throughout the article, we fix the convention that $\{(i,j) \in \Z^2: (i+j)\mod 2 = 0\}$ is the set of even vertices, and that homomorphisms take even values on even vertices.  

When $a=b=1$ and $c$ is arbitrary, the logarithm of the probability of $h$ is proportional to the number of diagonally connected vertices $u$ and $v$ for which $h_u=h_v$. In particular, when $c=1$, which corresponds to the famous {\em square-ice} model, the distribution of $h$ is the uniform measure. 

The six-vertex model became the archetypical example of an integrable model after Lieb's solution of the model in 1967 in its anti-ferroelectric and ferroelectric phases \cite{Lie67a,Lie67b,Lie67c} using the Bethe ansatz (see \cite{BetheAnsatz1} and references therein for a review). Since its exact solution, the model has been intensively studied, yet most of the results had fallen short of addressing the question of localization/delocalization of the associated height function. The situation changed in the last two decades. The model at its free fermion point (i.e.~when $c=\sqrt 2$) was directly related to the dimer model, and the height function was proved to converge to GFF (see \cite{BK} and reference therein). For $c\ge \sqrt3$, the model is related to the critical random-cluster models with $q\ge1$, where a discontinuous/continuous phase transition was proved in  \cite{DGHMT} and \cite{DST} for $q>4$ and $1\le q\le 4$, respectively (see also \cite{RS19}). This immediately implies that the associated height function of the six-vertex model is localized for $c>2$ and delocalized for $c= 2$ (see, for example, \cite{glazman2019transition}). Finally \cite{CPT18}, borrowing ideas from \cite{She05} proved that the square-ice height function is delocalized. 

In this paper, we wish to study the behaviour of the height function in the delocalized phase. 
We start by the following result. Let $\phi_{\mathbb T_n}$ be the uniform distribution for height function on the torus\footnote{Technically, such a distribution is only defined up to a translation --- we will assume $h_u = 0$ for some fixed $u$, and note that all terms in the theorem below do not depend on the choice of $u$.}. 

\begin{thm}\label{thm:torus}
There exist $c,C\in(0,\infty)$ such that for every $n\ge1$ and every $u,v\in \mathbb T_n$,
$$c\log \|u-v\|_1\le \phi_{\mathbb T_n}[(h_u-h_v)^2]\le C\log \|u-v\|_1,$$
where $\|\cdot \|_1$ is the $L^1$ distance in $\mathbb T_n$.
\end{thm}

In order to state the main result of this paper, we need some more notation. 
A {\em path} (\emph{$\times$-path}) is a sequence of vertices $v_0,\dots,v_n$ in $\Z^2$ such that for every $0\le i<n$, $v_i$ and $v_{i+1}$ are at a Euclidean distance 1 (resp.~$\sqrt 2$) of each other. When $v_n=v_0$, we call a $\times$-path a {\em circuit}. {Throughout the paper, we will assume that $(v_i, v_{i+1}) \neq (v_j, v_{j+1})$ whenever $i \neq j$, and that the ordering is chosen so that the induced continuous circuit\footnote{The continuous path is made by joining the vertices by straight lines  in $\mathbb{R}^2$.} is not self-crossing (though it may be self-touching).} We will often use the notation $[v_iv_j]$ (resp.~$(v_iv_j)$) for the subpath of the path made of the vertices $v_i,\dots,v_j$ (resp.~$v_{i+1},\dots,v_{j-1}$). {Given a circuit, we define a {\em domain} $D$ to be the finite subset of $\Z^2$ that is enclosed by the circuit; if the circuit is made up of even (resp.~odd) vertices, we call the domain an {\em even} (resp.~{\em odd}) {\em domain}. We will denote the circuit which defines the domain $D$ by $\partial D$; we will refer to this circuit as the boundary of the domain.  A {\em quad} $(D,a,b,c,d)$ is given by a domain with four marked points in $\partial D$, appearing in the order induced by the circuit. Throughout the paper, we refer to certain subgraphs of $\Z^2$ colloquially as `domains' even though the subgraphs may not precisely coincide with the description here. In each such case, the subgraph is explicitly described so there should be no confusion.}

%


For a quad $(D,a,b,c,d)$, the event $\cC_{h\in I}(D,a,b,c,d)$ is the event that there exists a path of vertices in $D$ with height in $I$ that connects $[ab]$ to $[cd]$ 
(we emphasize that $[ab],[cd]$ are $\times$-paths). 
We  use the shortcut $h=m$, $h\ge m$ and $|h|>0$ when $I=\{m\}$, $I=[m,\infty)$ and $I=\Z\setminus\{0\}$. We extend this definition to $\times$-paths by introducing the notation $\cC_{h\in I}^\times(\cD)$. 
When $D$ is a rectangle $R$, we introduce $\cH_\#^\#(R):=\cC_\#^\#(R,a,b,c,d)$ and $\cV_\#^\#(R)=\cC_\#^\#(R,b,c,d,a)$, where $a,b,c,d$ are the four corners of $R$ indexed in counter-clockwise order starting from the top-left one, corresponding to the existence of horizontal and vertical crossings of the rectangle. 
{ We write $\Lambda_{x,y}$ for $\Lambda_{\lfloor x \rfloor, \lfloor y \rfloor} $ for $x,y \in \R$ and $\Lambda_x$ in short for $\Lambda_{x,x}$.}

Let $\phi_D^0$ be the uniform distribution on height functions defined on an even domain $D$ which are equal to 0 on $\partial D$. This model corresponds to the height function of square-ice when the arrow configurations are defined on the set $E^*$ of edges of $(\Z^2)^*$ bordering the faces of $(\Z^2)^*$ centred on vertices in $D$, and one applies the {\em generalized ice-rule} stating that every vertex has the same number of incoming and outgoing arrows.  
We consider two possible behaviours:
\begin{itemize} 
\item[B1] There exist  $C,c>0$ such that  for every $k$ and every even domain $D$
 $$\phi_{D}^0[|h_0|>k]\le C\exp(-k^c).$$
\item[B2] {For every $\varepsilon,R,\rho,k>0$, there exists $c=c(\varepsilon,R,\rho,k)>0,C = C(\rho,k)>0$ such that for every $n \ge C$ and every even domain $D\subset \Lambda_{Rn}$ such that the distance between $\Lambda_{\rho n,n}$ and $\partial D$ is at least $\varepsilon n$,
 \begin{align}
c~\le~ &\phi_D^0 [\cH_{h\ge k}(\Lambda_{\rho n,n})]~\le~1- c,\label{eq:ha1}\\
c~\le~&\phi_D^0 [\cH_{h=k}^\times(\Lambda_{\rho n,n})]~\le~1-c.\label{eq:ha2}
\end{align}
}
\end{itemize}

The first case corresponds to a strongly localized behaviour, while the second one corresponds to a delocalized one. For instance, we will see that B2 implies Theorem~\ref{thm:torus} very easily. In fact, it also implies that $\phi_D^0[h_0^2]$ is growing logarithmically in the distance to $\partial D$. Let us mention that it also easily implies tightness of the family of uniformly chosen height functions when taking the scaling limit of the model, a fact which may be useful to prove convergence to GFF. 

We now state what we consider to be the main contribution of this paper. 
 \begin{thm}\label{thm:dichotomy}
For the height function of square-ice, either B1 or B2 occurs.
\end{thm}
We insist on stating this result as a dichotomy between two possible behaviours since we believe that this result can be extended to more general random height functions (see Question~\ref{q1} below). Nevertheless, the result of \cite{CPT18,She05} excludes B1, so that we get the following immediate corollary.
\begin{corollary}\label{cor:main}
For the height function of square-ice, B2 occurs.
\end{corollary}
At a high level, our strategy to prove Theorem~\ref{thm:dichotomy} follows \cite{DST}, with some inspiration from \cite{DT19}. It is based on a renormalization argument (which is made more complicated by the height-function structure, see the discussion in Section~\ref{sec:4.2}) and a Russo-Seymour-Welsh (RSW) theory for height functions. The RSW theory is a study of probabilities of crossing events in planar percolation models. This theory was initially created for the study of Bernoulli percolation \cite{Rus78,SeyWel78,BolRio06c}. It has blossomed in the past decade and now applies to a wide variety of percolation models \cite{BolRio10,BefDum12,Tas15,DumHonNol11,DST,DGPS,GM}. In this paper, we provide the first adaptation of the theory to the study of random height functions. 

\begin{thm}\label{thm:RSW}
For every $\rho>0$, there exists $c=c(\rho)>0$ and $C =C(\rho)>0$ such that for every $n \ge C(\rho)$ and every even domain $D$ containing $\Lambda_{\rho n,n}$,
\begin{equation}
\phi_{\overline D}^0[\cH_{h\ge 2}^\times(\Lambda_{\rho n,n})]~\ge~c\phi_{D}^0[\cV_{h\ge 2}^\times(\Lambda_{\rho n,n})]^{1/c},
\end{equation}
where $\overline D$ is an even domain containing all the translates of $D$ by $(k,0)$ with $|k|\le 4\rho n$.
\end{thm}
The previous theorem is called a RSW theorem in the sense that it bounds the probability of crossing rectangles in the `hard' direction in terms of the probability of crossing rectangles in the `easy' direction.
With a little more work, one may replace the right-hand side by a quantity that tends to 1 when the probability of a vertical crossing tends to 1. We will also see that the theorem adapts trivially to other geometry, such as the infinite strip. 

%

\subsection{Related results and open questions}\label{sec:1.2}

The uniform measure on homomorphisms was also introduced independently of the six-vertex model by
Benjamini, H\"aggstr\"om and Mossel in \cite{BHM_homo} (see also \cite{BP94} for a prior work focusing on the tree)
and further investigated in \cite{BS00,K01,LNR03,G03,BYY07,E09,P17} on arbitrary graphs (for which there is no {\em a priori} connection to square-ice). As mentioned above \cite{She05,CPT18}, the model is delocalized on $\Z^2$. In fact, the model undergoes a roughening phase transition; in \cite{P17}, it was proved that, for every $k\ge2$ and sufficiently large $d$, the height function on $\mathbb T_n\times(\Z/k\Z)^d$ is {\em localized}.

Related studies consider the behaviour of the class of integer-valued $1$-Lipschitz functions. When the base graph is the triangular lattice, delocalization and logarithmic variance has been established through a correpondence with the loop O(2) model with edge weight $x=1$ \cite{GM}. Another delocalization result is obtained in \cite{DGPS} for the height function of the loop O(2) model with weight $1/\sqrt 2$. 

It is natural to ask to which extent the techniques developed in this paper help understand the height function of the six-vertex model for different values of $c$. We believe that one of the main contributions of the paper lies in the use of the FKG inequality for $|h|$ to implement comparison between boundary conditions and obtain the RSW theory and the renormalization for crossing probabilities. This FKG for $h$ and $|h|$ are valid for every six-vertex model with $a=b=1$ and $c\ge1$. We therefore believe that an argument similar to the present paper could lead to an equivalent of Theorem~\ref{thm:dichotomy} in the regime $a=b=1$ and $c\ge1$. This would be particularly interesting since some range of $c\ge1$ corresponds to the random-cluster model with $q\in(0,1)$ which is known not to satisfy the FKG inequality as a percolation model. Unfortunately, the present techniques do not extend in a trivial fashion due to the lack of spatial Markov property when $c>1$. More precisely, take the example of an even domain. For $c=1$, the value of $h$ on $\partial V$ is sufficient to decorrelate the outside from the inside, while this is no longer the case for $c>1$ (one needs to know what are the values in diagonals as well). For this reason, we leave the following interesting problem open.
\begin{question}\label{q1}
Prove Theorem~\ref{thm:dichotomy} for the height function of the six-vertex model with $a=b=1$ and $c\ge1$.
\end{question}
Of course, we do not address the important open question of proving GFF fluctuations.
\begin{question}
Prove that the square-ice height function in an even domain $\Omega_\delta\subset\delta\Z^2$ approximating a simply connected open set $\Omega$ converges weakly to the GFF on $\Omega$ with Dirichlet boundary condition 0 on $\partial\Omega$.
\end{question}

\paragraph{Organisation of the paper} Section~\ref{sec:2} contains some preliminaries (FKG and duality properties). Section~\ref{sec:3} deals with the proof of Theorem~\ref{thm:RSW} while Section~\ref{sec:4} presents the proof of the other theorems.

 \paragraph{Acknowledgments} 
The first author is supported by the ERC CriBLaM, the NCCR SwissMAP, the Swiss NSF and an IDEX Chair from Paris-Saclay. The second author was supported in part by the European Research Council starting grant 678520 (LocalOrder), and the Zuckerman STEM leadership Postdoctoral Fellowship. The last author is supported in part by NSERC 50311-57400. This project was initiated during the visit of the last author in IHES. The authors would like to express their gratitude to IHES for its support. Finally, we thank Alex Karrila and the anonymous referee for carefully reading the manuscript.

 \section{Preliminaries}\label{sec:2}
 In this section, we gather some simple facts about homomorphisms. More precisely, the first part proves the FKG inequality while the second discusses certain connectivity issues that will be important in the following sections.

 In order to properly state these properties, we introduce a general notion of boundary condition. For $B\subset D$ with $ D$ a domain and $\kappa$ a function from $B$ into the subsets of $\Z$, define ${\rm Hom}( D,B,\kappa)$ to be the set of homomorphisms $h$ on $ D$ such that $h_v\in \kappa_v$ for every $v\in B$. { We emphasize that $B$ need not be a subset of $\partial D$ in this definition. For $a,b \in \R$, we use the notation $\kappa_v = [a,b]$ for $\kappa_v  = [a,b] \cap \Z$.} We call $(B,\kappa)$ a {\em boundary condition} and say that the boundary condition is {\em admissible} if ${\rm Hom}( D,B,\kappa)\ne \emptyset$ is finite. For admissible boundary condition $(B,\kappa)$, we set $\phi^{B, \kappa}_{ D}$ for the uniform measure on ${\rm Hom}( D,B,\kappa)$.  When $B=\partial D$, we drop it from the notation.

  \subsection{Monotonicity properties of uniform homomorphisms}\label{sec:2.1}
We call a function $F : \Z^D \mapsto \R$ \emph{increasing} if for any $h, h' \in   \Z^D$  satisfying $h_v \ge h'_v$ for all $v \in  D$, $F(h) \ge F(h')$. 

 \begin{prop}[monotonicity for $h$]\label{prop:FKG_h}
 Consider $D\subset \Z^2$ and two admissible boundary conditions $(B,\kappa)$ and $(B,\kappa')$ satisfying that for every $v\in B$, $\kappa_v=[a_v,b_v]$ and $\kappa'_v=[a'_v,b'_v]$ with $a_v\le a'_v$ and $b_v\le b'_v$ (the previous integers may be equal to $\pm\infty$), then \begin{description}
 \item[(CBC)] For every increasing function $F$, 
 $\phi^{B, \kappa'}_D[F(h)]\ge\phi^{B, \kappa}_D[F(h)]$;
 \item[(FKG)] For any two increasing functions $F,G$, $\phi^ {B, \kappa}_D[F(h)G(h)] \ge   \phi^ {B, \kappa}_D[F(h)]  \phi^ {B, \kappa}_D [G(h)]$.
 \end{description}
 \end{prop}
 The first property is called the {\em comparison between boundary conditions}, and the second the {\em Fortuin-Kasteleyn-Ginibre (FKG)} inequality.
\begin{proof}
These results follow from Holley's criterion,  see \cite[Theorem 4.8]{GHM01}, since our definition of height function specifies even height on the even sublattice and therefore implies irreducibility (a fact which is required for Holley's criterion). Note that the conditions on the boundary conditions are designed so that Holley's criterion holds for boundary vertices.
\end{proof}

We also crucially use monotonic properties for $|h|$ instead of $h$. In order to properly state the conditions for such an inequality, we introduce some new notation. We say that the boundary condition $\kappa$ is {\em $|h|$-adapted} if there exists a partition $B_{\rm pos}(\kappa) \sqcup B_{\rm sym}(\kappa)$ of $B$ such that
\begin{itemize}[noitemsep,nolistsep]
\item for any $v \in B_{\rm pos}(\kappa)$, {$\kappa_{v}\subset\Z_+:=\{1,2,\dots\}$};
\item for any $w \in B_{\rm sym}(\kappa)$, $\kappa_w =-\kappa_w$. 
\end{itemize}

Let ${\rm Hom}^{\ge 0}(D,B,\kappa)$ be the set of all $\xi \in {\rm Hom}(D,B,\kappa)$ with $\xi_v \ge 0$ for all $v \in D$. Let us first make the following simple observation. For any $\xi \in {\rm Hom}^{\ge 0}(D,B,\kappa)$ where $\kappa$ is a $|h|$-adapted boundary condition, let $k(\xi)$ denote the number of connected components (with regular connectivity of $\Z^2$) of the set of sites with $\xi_v \ge 1$ which do not intersect $B_{\rm pos}$. Note that the sign of $h$ is the same for each such cluster and could be either $+$ or $-$ with equal weight. Consequently
\begin{equation}
\phi^{B, \kappa}_D (|h| = \xi) =\frac{2^{k(\xi)}}{|{\rm Hom}(D,B,\kappa)|}\label{eq:simple_obs}
\end{equation}
where $|\cdot|$ in the denominator denotes the cardinality of the set. 

   \begin{prop}[monotonicity for $|h|$]\label{prop:FKG_modh}
    Consider $D\subset\Z^2$ and two admissible $|h|$-adapted boundary conditions $(B,\kappa)$ and $(B,\kappa')$ satisfying $B_{\rm pos}(\kappa) \subseteq B_{\rm pos}(\kappa')$ and for every $v\in B$, $[a_v,b_v]:=\kappa_v\cap(\Z_+ \cup \{0\}) $ and $[a'_v,b'_v]:=\kappa'_v\cap(\Z_+ \cup\{0\})$ satisfy $a_v\le a'_v$ and $b_v\le b'_v$,
\begin{description}
 \item[(CBC)] For every increasing function $F$, 
 $\phi^{B, \kappa'}_D[F(|h|)]\ge\phi^{B, \kappa}_D[F(|h|)]$;
 \item[(FKG)] For any two increasing functions $F,G$, $\phi^ {B, \kappa}_D[F(|h|)G(|h|)] \ge   \phi^ {B, \kappa}_D[F(|h|)]  \phi^ {B, \kappa}_D[G(|h|)]$.
 \end{description}
 \end{prop}
\begin{proof}
Fix a vertex $v $. Using \cite[Theorem 4.8]{GHM01} and references therein, { it is sufficient to prove that for any $\xi$ (resp. $\eta$) which are restrictions to $D \setminus \{v\}$ of configurations  in ${\rm Hom}^{\ge 0}(D,B,\kappa)$ (resp. ${\rm Hom}^{\ge 0}(D,B,\kappa')$)  such that $\xi_v \le \eta_v$ for all $v$}, and every $k  \ge0$,
\begin{equation}
\phi_{D}^{B,\kappa}\big[|h_v| \ge k\big| |h_{|D\setminus\{v\}}|= \xi\big] \le \phi_{D}^{B,\kappa'}\big[|h_v| \ge k
 \big| |h_{|D\setminus\{v\}}|= \eta \big]\label{eq:Holley_cond}.
 \end{equation}

Suppose first that $\xi_u \neq \xi_{u'}$ for two neighbours of $u \neq u'$ of $v$. In this case because of parity constraints they must differ by $2$ and the value of $\xi_v$ is deterministic. Also if $\xi_u = m$ for all neighbours of $v$ with $m \ge 2$, the conditional distribution of $|h_v|$ on the left hand side of \eqref{eq:Holley_cond} is equally likely to be $m-1$ or $m+1$ if both $\{m-1,m+1\}$ is in $\kappa_v$(using \eqref{eq:simple_obs}). Indeed $k(\xi)$ defined in \eqref{eq:simple_obs} does not change if $\xi_v = m-1$ or $m+1$ (even if $v \in B_{\rm pos}$). On the other hand if exactly one of $m-1$ or $m+1$ is in $\kappa_v$, then $|h_v|$ is deterministically that value. Using these two facts, it is easy to verify \eqref{eq:Holley_cond} unless $\xi_u = \eta_u=1$ for all neighbours $u$ of $v$.

  Thus, let us now consider the case $\xi_u= \eta_u = 1$ for every neighbour $u$ of $v$. {Define $k_v(\xi)$  (resp. $k_v(\eta)$) to be the number of connected components in $D\setminus \{v\}$ of $\{u:\xi_u \ge 1\}$ (resp. $\{u:\eta_u \ge 1\}$) containing at least one neighbour of $v$ and that is {\em not} intersecting $B_{\rm pos}(\kappa)$ (resp. $B_{\rm pos}(\kappa'))$. Let $\cB(\kappa) $ be the event that at least one of the components of $\{u:\xi_u \ge 1\}$ containing at least one neighbour of $v$ intersects $B_{\rm pos}(\kappa)$ and similarly define $\cB(\kappa')$. Note that since $B_{\rm pos}(\kappa) \subseteq B_{\rm pos}(\kappa')$ and $\xi_v \le \eta_v$ for all $v$, $\cB(\kappa) \subseteq \cB(\kappa')$. } If $\{0,2\} \not \subseteq \kappa_v$ or $\{0,2\} \not \subseteq \kappa'_v$, the value at $v$ is deterministic and we are done by conditions on $\kappa, \kappa'$. In the other cases, it is easy to deduce from \eqref{eq:simple_obs}
\begin{equation*}
\phi_{D}^{B,\kappa}\big[|h_v| =0\big| |h_{|D\setminus\{v\}}|= \xi\big] =
\begin{cases}
0 \text{ if }v \in B_{\rm pos}\\
 2^{k_v(\xi)} /(2 + 2^{k_v(\xi)}) \text{ if  $\xi \notin \cB(\kappa)$, $v \notin B_{\rm pos}$}\\
 2^{k_v(\xi)} /(1 + 2^{k_v(\xi)})\text{ if  $\xi \in \cB(\kappa)$, $v \notin B_{\rm pos}$}
\end{cases}
\end{equation*}

and the same is true if we replace $\xi$ by $\eta$ and $\kappa$ by $\kappa'$. In the first case $v \in B_{\rm pos}$ we are immediately done by the definition of $|h|$-adapted boundary condition. Observe that $k_v(\xi) \ge k_v(\eta)$ since $\eta _v \ge \xi_v$ for all $v$. If there is strict inequality $k_v(\xi) > k_v(\eta)$ then the above expression for $\xi$ is at least that for $\eta$ in all the cases by monotonicity of $x \mapsto x/(1+x)$ and $x\mapsto x/(2+x)$ and the fact that $2^{k}/(2+2^k)\ge 2^{k'}/(1+2^{k'}) $ if $k\ge k'+1$. So assume $k_v(\xi) = k_v(\eta)$. If $\xi, \eta$ are both in $\cB(\kappa)$ or both not in $\cB(\kappa)$, we are also done for similar reasons. So the only case remaining is if $\eta \in \cB(\kappa')$ but $\xi \notin \cB(\kappa)$ (recall $\cB(\kappa) \subseteq \cB(\kappa')$). But in this case, the strict inequality $k_v(\xi) > k_v(\eta)$ must hold as there is a neighbour of $v$ whose cluster intersects $B_{\rm pos}(\kappa')$ and is not counted in $k_v(\eta)$ but every neighbour cluster is counted in $\xi$ as none of them intersect $B_{\rm pos}(\kappa)$. Thus we are back to the previous case, and the proof is complete.
%
\end{proof}
\begin{remark}The non-trivial case of the proof above is reminiscent of the proof of FKG for the FK-Ising model and the condition $B_{\rm pos}(\kappa) \subseteq B_{\rm pos}(\kappa')$ is equivalent to ``wiring'' more subsets of the boundary. In fact, the proof can be generalized to a case in which the boundary condition specifies an arbitrary `wiring' -- i.e.~forcing an arbitrary partition of the boundary to take on the same sign without choosing the particular sign. 
\end{remark}

 \subsection{Connectivity properties of lattice paths}\label{sec:2.2}
Our analysis will deal with paths of vertices in the square lattice and will crucially rely on the property that if a certain path does not connect two arcs of a quad, then there must exist a blocking path connecting the two other arcs. The study will be complicated here by the fact that these blocking paths will not necessarily be of the same kind as the original paths. We therefore gather a few technical statements to which we will refer in the next sections. 

\begin{figure}[h]
\centering
\includegraphics[width = .4 \textwidth]{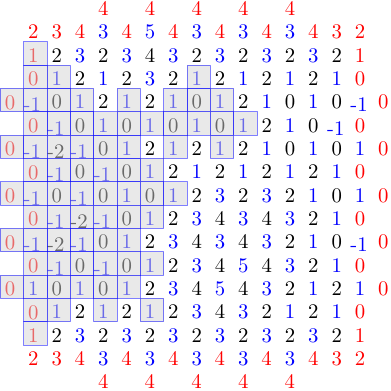}
\caption{Duality in a square. A top to bottom $h \ge 2$ $\times$-path blocks a left to right $h \le 1$ path (the cluster of $h \le 1$ containing the left boundary is shaded). This is a square with symmetric boundary condition (depicted in red) so that the top and bottom boundaries have value 4 and the left and right boundaries have value 0, appropriately modified at the corners. By this duality and symmetry, in a uniform homomorphism, a top to bottom $h \ge 2$ $\times$-path occurs with probability at least $1/2$.}\label{F:square}
\end{figure}

It will be convenient for proofs to introduce a notion of connectivity which is dual to the $\times$-paths. We will say that a path is a $*$-path if successive vertices are at graph distance exactly 2 of each other on $\Z^2$. We introduce the events 
$\cC_{h\in I}^*(D)$ with the notion of $*$-path.
The proof of the following lemma is straightforward and left as an exercise (see \cref{F:square} for an illustration).

\begin{lemma}\label{lem:duality}
For a quad $( D,a,b,c,d)$ and $m\in \Z$, we have the following properties
\begin{itemize}
\item We have  $$\cC_{h>m}^\times( D,a,b,c,d)^c=\cC_{h\le m}(D,b,c,d,a)=\cC_{h<m}^*( D,b,c,d,a)\supset \cC_{h<m}^\times( D,b,c,d,a).$$
{ If $[bc], [da]$ are paths with the same parity as $m$, the notation $\cC_{h<m}^*( D,b,c,d,a)$ means that there is a $*$-path joining two neighbours of $[bc],[da]$ inside $D$.}
\item if $(\partial D,\kappa)$ is an admissible boundary condition which satisfies $\kappa_v \subset[k,\infty]$ for each $v \in [ab] \cup [cd]$ and $\kappa_v \subset[-\infty,k]$ for each $v \in [bc] \cup [ad]$, then for any $m \ge k $, on ${\rm Hom}( D,\partial  D,\kappa)$,
\begin{align}
 \cC_{h\ge m}( D,a,b,c,d)&=\cC_{h\in\{m,m+1\}}( D,a,b,c,d)=\cC_{h=m+1}^*( D,a,b,c,d), \\
   \cC_{h\ge m}^\times( D,a,b,c,d)&=\cC_{h=m}^\times( D,a,b,c,d);
      \end{align}
      \item If $m$ is further assumed to be strictly positive, 
$$
 \cC_{|h|\ge m}( D,a,b,c,d)=\cC_{h\ge m}( D,a,b,c,d)\cup\cC_{h\le -m}( D,a,b,c,d). 
$$
\end{itemize}
\end{lemma}

\begin{remark}The last item  tells us that the existence of a $h \geq m$ crossing is {\em nearly} measurable with respect to the absolute value for any $m \geq 1$. Indeed, $|h|$ determines the connected structure of $h  \neq 0$, up to the sign of each cluster. Now, if ${\rm Hom}( D,B,\kappa)$ is chosen in a manner that determines the sign of a crossing from $[ab]$ to $[cd]$, the event becomes truly measurable with respect to $|h|$. Note that this property does not generalize to every type of connections: while $\times$-crossings of $h \geq 2$ or $h \le -2$  can be decided by $|h|$, the same is not true for $\times$-crossings of $h \geq 1$ (or $h  \le -1$) since $\times$-neighbours may have different signs.
\end{remark}

\section{Russo-Seymour-Welsh theory}\label{sec:3}
In this section, we prove Theorem~\ref{thm:RSW}. In Section~\ref{sec:3.1}, we start by presenting the proof subject to two propositions that we prove in Sections~\ref{sec:3.2} and \ref{sec:3.3}.

\subsection{Proof of Theorem~\ref{thm:RSW}}\label{sec:3.1}

We prove the result for the rectangle $\Lambda_{3\rho n,3n}$. We introduce the rectangles 
\begin{align*}
R_n^-&=[-3\rho n,3\rho n]\times[-3n,-n],\\
R_n^0&=[-3\rho n,3\rho n]\times[-n,n],\\
R_n^+&=[-3\rho n,3\rho n]\times[n,3n].
\end{align*}
For $\ep<\min\{1/11, \rho/100\}$ and $k$, we set the notations (we keep the dependence on $n$ hidden in the notations)
\begin{align*}
I_k&:=[(2k-1)\ep n, (2k+1)\ep n] \times \{-3n\},\\
J_k&:=[(2k-1)\ep n, (2k+1)\ep n] \times \{-n\},\\
K_k&:=[(2k-1)\ep n, (2k+1)\ep n] \times \{n\},\\
L_k&:=[(2k-1)\ep n, (2k+1)\ep n] \times \{3n\}.
\end{align*}
Let us start by a simple observation that will motivate our reasoning below. Set $\cA^i$ to be the event that $I_{i}$ and $I_{i+2}$ are connected by a $\times$-path of $|h|\ge2$ staying between heights $-3n$ and $3n$. Recall that $\overline D$ is an even domain containing all the translates of $D$ by $(k,0)$ with $|k|\le 4\rho n$.
 The $\pm$-symmetry and the FKG inequality for $|h|$ implies that  
\begin{equation}\label{eq:push0}
2\phi_{\overline D}^0[\cH_{h\ge2}^\times(\Lambda_{3\rho n,3n})]\ge\phi_{\overline D}^0[\cH_{|h|\ge2}^\times(\Lambda_{3\rho n,3n})]\ge \prod_{j=-\lceil3\rho/\ep\rceil-1}^{\lceil3\rho/\ep\rceil}\phi_{\overline D}^0[\cA^j]. 
\end{equation}
Furthermore, for every $i_0$ and $i$, the FKG inequality for $|h|$ implies that
\begin{equation}\label{eq:push1'}
\phi_{\overline D}^0[\cA^{i_0}]\ge\phi_{\overline D}^0[\cA^{i_0}|h_{|\partial\widetilde D^{i_0-i}}=0]=\phi_{\widetilde D^{i_0-i}}^0[\cA^{i_0}]=\phi_{\widetilde D}^0[\cA^i], 
\end{equation}
where $\widetilde D$ is the union of the translations of $D$ by $(4k\ep n,0)$ with $-2\le k\le3$ and $\widetilde D^i$ is the translate by $(2\ep i,0)$ of $\widetilde D$. The reason for introducing $\widetilde D$ will become clear after \eqref{eq:construction}. Therefore, our goal is to find a constant $c = c(\rho)$ such that for all $n \ge 1$
\begin{equation}
\max_i \phi_{\tilde D}^0 (\cA^i) \ge c\phi_{D}^0[\cV_{h\ge 2}^\times(\Lambda_{3\rho n,3n})]^{1/c} \label{eq:goal}
\end{equation}
as then for every $j$ in the product of \eqref{eq:push0}, we can apply \eqref{eq:push1'} with $i_0 =j$ and $i=$ the index which attains the max in \eqref{eq:goal} to obtain the desired result.
\bigbreak
In order to prove \eqref{eq:goal}, let $\cE_{ijk\ell}$ be the event that there is a vertical $\times$-crossing of $|h|\ge2$ in $\Lambda_{\rho n,3n}$ that starts from $I_i$ and ends at $L_\ell$, and which contains a sub-path crossing going from $J_j$ to $K_k$ in $R_n^0$. We further define $\cE_{i}$ to be the event that there is a vertical $\times$-crossing of $|h| \geq 2$ to $\mathbb{Z} \times \{3n\}$ with no further restrictions on the geometry. 

\begin{figure}[h]
\centering
\includegraphics[width = 0.5\textwidth]{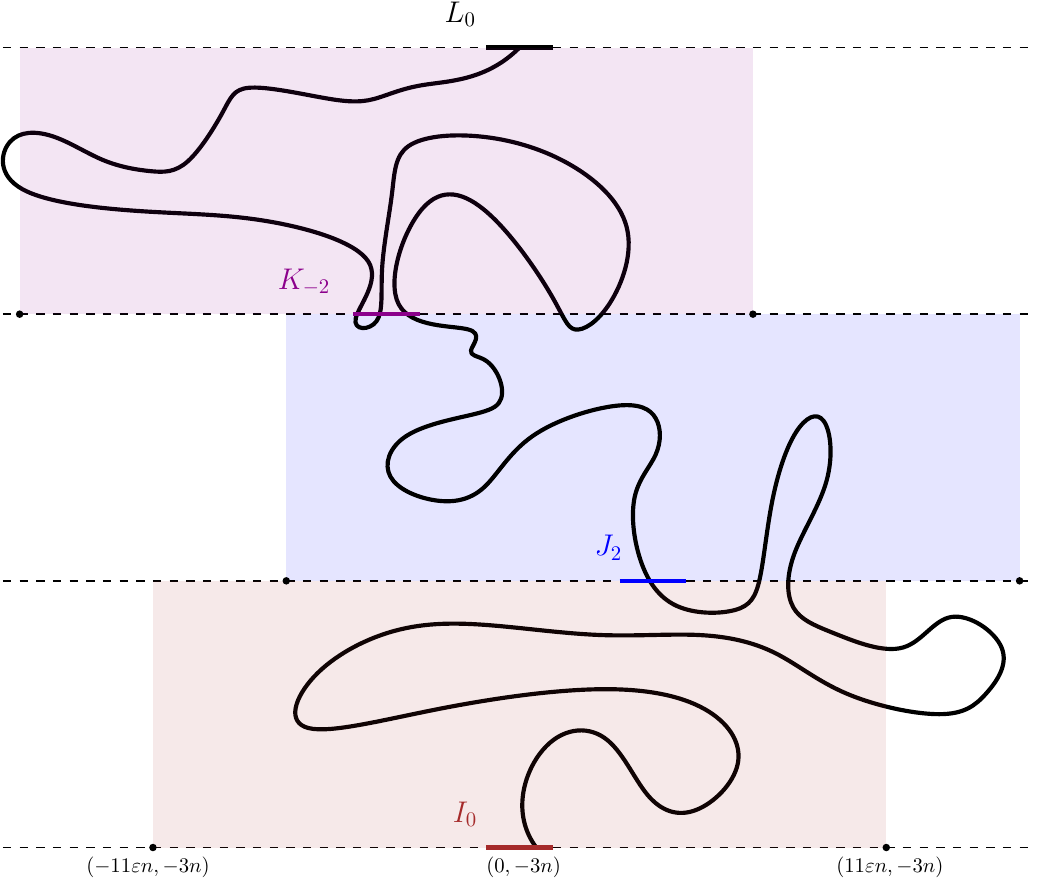}
\caption{A vertical crossing which realizes the event $\cE^{\alpha\beta\gamma}_{ijk\ell}$ with $\alpha = -, \beta = 0, \gamma = +$ and $i=0, j= 2, k= - 2, \ell = 0$.}\label{fig:eijkl}
\end{figure}

For $\alpha,\beta,\gamma\in\{-,0,+\}$, introduce the event $\cE_{ijk\ell}^{\alpha\beta\gamma}$ that $\cE_{ijk\ell}$ occurs and in the $\times$-cluster of $|h|\ge2$ in $\Lambda_{\rho n,3n}$ starting from $I_i$, one can find  (see \cref{fig:eijkl})
\begin{itemize}[noitemsep]
\item a vertical $\times$-crossing of $R_n^-$ starting from $I_i$ and staying in $[(2i-11)\ep n,(2i+11)\ep n]\times\Z$ (resp.~intersecting $\{(2i-11)\ep n\}\times\Z$  or $\{(2i+11)\ep n\}\times\Z$) if $\alpha=0$ (resp.~$\alpha=+$ or $\alpha=-$);
\item a vertical $\times$-crossing of $R_n^0$ starting from $J_j$ and staying in $[(2j-11)\ep n,(2j+11)\ep n]\times\Z$ (resp.~intersecting $\{(2j-11)\ep n\}\times\Z$  or $\{(2j+11)\ep n\}\times\Z$) if $\beta=0$ (resp.~$\beta=+$ or $\beta=-$);
\item a vertical $\times$-crossing of $R_n^+$ starting from $K_k$ and staying in $[(2k-11)\ep n,(2k+11)\ep n]\times\Z$ (resp.~intersecting $\{(2k-11)\ep n\}\times\Z$  or $\{(2k+11)\ep n\}\times\Z$) if $\gamma=0$ (resp.~$\gamma=+$ or $\gamma=-$).
\end{itemize}
The square-root trick\footnote{We prefer the use of the square-root trick to the use of the union bound since we will refer to this argument later with events having a probability close to 1. We recall that the square-root trick yields that for increasing events $\cA_,\dots,\cA_s$ and a measure $\mathbb  P$ satisfying the FKG inequality,
$$\max_{i\le s}\mathbb P[\cA_i]\ge 1-(1-\mathbb P[\cA_1\cup\dots\cup \cA_s])^{1/s}.$$} implies that there exist $i,j,k,\ell$ and $\alpha,\beta,\gamma$ such that 
\begin{equation}\label{eq:partition}\phi_D^0[\cE_{ijk\ell}^{\alpha\beta\gamma}]\ge 1-\Big(1-\phi_D^0[\cV_{|h|\ge2}^\times(\Lambda_{3\rho n,3n})]\Big)^{1/C}
,\end{equation}
where $C=C(\ep,\rho)\ge 27\lceil 3\rho/\ep\rceil^4$. From now on, we fix $i,j,k,\ell,\alpha,\beta,\gamma$ such that \eqref{eq:partition} holds, and set $\cE:=\cE_{ijk\ell}^{\alpha\beta\gamma}$ and $\bar{\cE} = \cE_{i}$ (so that $\cE \subset \bar{\cE}$, and the latter event does not restrict the geometry or the subpaths comprising the vertical crossing, except the initial intersection with $I_i$). 

If 
\begin{equation}
\max \{\phi_{\widetilde D}^0[\cA^{i-2}] , \phi_{\widetilde D}^0[\cA^{i+2}] \} \ge  \tfrac1{3}\phi_D^0[\cE]^4 \label{eq:event_max}
\end{equation}
then by \eqref{eq:partition}
\begin{equation}
(3\max_i \phi_{\widetilde D}^0[\cA_i])^{1/4} \ge 1-\Big(1-\phi_D^0[\cV_{|h|\ge2}^\times(\Lambda_{3\rho n,3n})]\Big)^{1/C}
\end{equation}
which implies \eqref{eq:goal}.

For the remainder of the proof, assume \eqref{eq:event_max} does not hold.  We also introduce the translate $\cE^k$ and $\bar{\cE}^k$ of $\cE$ and $\bar{\cE}$ by $(2k\ep n,0)$. First note that 
\begin{equation}
\phi_{\widetilde D}^0[\cE^{-2}\cap\bar{\cE}\cap\bar{\cE}^2\cap\cE^4\cap(\cA^{i-2})^c\cap(\cA^{i+2})^c] \ge \phi_{\widetilde D}^0[\cE^{-2}\cap\bar{\cE}\cap\bar{\cE}^2\cap\cE^4]  - \phi_{\widetilde D}^0[\cA^{i-2} \cup \cA^{i+2}]
\end{equation}
The FKG inequality for $|h|$ implies that, as in \eqref{eq:push0},  
\begin{equation}\phi_{\widetilde D}^0[\cE^{-2}\cap\bar{\cE}\cap\bar{\cE}^2\cap\cE^4]\ge \phi_D^0[\cE]^4.\label{eq:construction}
\end{equation}
We remark that the domain $\widetilde D$ was introduced precisely to make this inequality manifest.  Therefore the negation of \eqref{eq:event_max} implies that 
\begin{equation}
\phi_{\widetilde D}^0[\cE^{-2}\cap\bar{\cE}\cap\bar{\cE}^2\cap\cE^4\cap(\cA^{i-2})^c\cap(\cA^{i+2})^c]  \ge \frac13 \phi_D^0[\cE]^4 \label{eq:big_intersection}
\end{equation}

The rest of the proof will be devoted to the proof of the following inequality:
\begin{equation}
\phi_{\widetilde D}^0[\cA^{i}|\cE^{-2}\cap\bar{\cE}\cap\bar{\cE}^2\cap\cE^4\cap(\cA^{i-2})^c\cap(\cA^{i+2})^c]\ge \tfrac1{32}. \label{eq:max_A2}
\end{equation}
Once this inequality is established, we can conclude the proof of the theorem, since it can be combined with \eqref{eq:big_intersection} and \eqref{eq:partition} to prove \eqref{eq:goal}.

In order to prove \eqref{eq:max_A2}, we first state two propositions that will be proved in Sections~\ref{sec:3.2} and \ref{sec:3.3} respectively.  Recall the definition of even quads and given an even quad $(D,a,b,c,d)$ we define $[ab]$ to be the path in the boundary $\times$-circuit connecting $a$ and $b$ and the same for $[bc],[cd],[da]$.
\begin{proposition}\label{prop:weak}
Fix $n\ge100$ and suppose $(D,a,b,c,d)$ is an even quad with the following properties. 
\begin{itemize}
\item $a,d \in \Z \times \{n\}$, $b,c \in \Z \times \{-n\}$,
\item $[ab] $ and $[cd]$ are in $\Z \times [-n,n]$.
\item  $[bc]$ and $[da]$ are in $\Z\times\{-n-1,-n\}$ and $\Z\times\{n,n+1\}$ respectively.
\end{itemize}
Set boundary condition $\kappa $ to be $2$ on $[ab]\cup[cd]$ and $0$ on $(bc)\cup(da)$. Then 
\[
 \phi_{D}^{\kappa}[\cC_{h\ge1}(D,a,b,c,d)]\ge \tfrac12,
\]
 if $D$ is in one of the following three configurations:
\begin{enumerate}[{(}i{)}]
\item $[ab]\cup[cd]$ is contained in $\Lambda_{n/2,n}$, 
\item $[ab]$ intersects the vertical line containing $c$,
\item $[cd]$ intersects the vertical line containing $b$.
\end{enumerate}
\end{proposition}

A quad $(D,a,b,c,d)$ is called {\em mixed} if $[ab]$ and $[cd]$ are even $\times$-paths, and $(bc)$ and $(da)$ are odd $\times$-paths. Note that in this case $D$ is not quite a domain according to the definition of the introduction, we will therefore refer to it as a {\em mixed-domain}  (see \cref{fig:mixed_domain}).

\begin{figure}
\centering
\includegraphics[width = 0.3\textwidth]{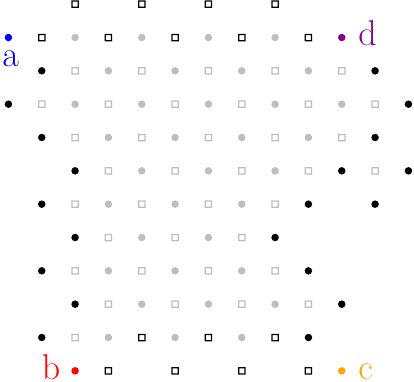}
\caption{A mixed domain; even vertices are marked with circles, while odd vertices are marked with boxes. The grey vertices indicate the interior.}\label{fig:mixed_domain}
\end{figure}

\begin{proposition}\label{prop:strong}
Fix $n\ge100$ and a mixed-quad $(D,a,b,c,d)$ with the following properties.
\begin{itemize}
\item $a,d \in \Z \times \{n\}$, $b,c \in \Z \times \{-n+1\}$. 
\item $[ab] $ and $[cd]$ are in $\Z \times \{-n+1,-n\}$. 
\item Let $H$ be the set of vertices enclosed by $[ab] $, $[cd]$, $\Z \times \{n\}$, and $\Z \times \{-n+1\}$. Then $(bc)$ and $(da)$  do not intersect $H$.
\end{itemize}
Set $\kappa$ to be the boundary condition equal to $2$ on $[ab]\cup[cd]$ and $1$ on $(bc)\cup(da)$. Then

\[
 \phi_{D}^{\kappa}[\cC_{h\ge2}^\times(D,a,b,c,d)]\ge \tfrac12,
\]
 if $D$ is in at least one of the following three configurations:
\begin{enumerate}[{(}i{)}]
\item $[ab]\cup[cd]$ is contained in $\Lambda_{n/2,n}$, 
\item $[ab]$ intersects the vertical line containing $c+(\tfrac12,0)$,
\item $[cd]$ intersects the vertical line containing $b-(\tfrac12,0)$.
\end{enumerate}
{ Here by `intersects' we mean the path obtained by linearly interpolating between the vertices intersect the vertical lines.}
\end{proposition}
We refer the reader to the top left figure of \cref{fig:new_hard_easy} for an example of a domain which satisfies the conditions of \cref{prop:strong}.

With these two propositions at hand, let us deduce \eqref{eq:max_A2}. The argument is divided in three steps: first, we transform our problem into the existence of a $\times$-crossing of $h\ge2$ in a domain with boundary condition 0/2/0/2. Then, we prove the existence of two crossings of $h\ge1$ in the `top $1/3$' and `bottom $1/3$' of the domain  using the first proposition twice. Finally, we use the second proposition to create a $\times$-crossing of $h\ge2$ in the domain enclosed by the top and bottom crossings of $1$. In each step, there is widespread  use of domain Markov property and FKG for $|h|$.
\bigbreak

Let us now come back to the proof of \eqref{eq:max_A2} which we restate here for the reader's convenience:
$$
\phi_{\widetilde D}^0[\cA^{i}|\cE^{-2}\cap\bar{\cE}\cap\bar{\cE}^2\cap\cE^4\cap(\cA^{i-2})^c\cap(\cA^{i+2})^c]\ge \tfrac1{32}. 
$$
where $\cE$ is the short form of $\cE_{ijkl}^{\alpha\beta\gamma}$ given by \eqref{eq:partition} and $\bar{\cE}$ is the short form of the (strictly larger) event that only requires $I_i$ to be connected to the top (with no additional geometric restriction).


{Let us now consider the leftmost- and rightmost-crossings which connect $I_i$ and $I_{i+2}$ to the top. We now explore $|h|$ from the left and right boundaries of $\Lambda_{3\rho n,3n}$ until we find these outermost crossings. We orient these paths so that the leftmost-crossing connecting $I_i$ to the top begins at $\mathbb{Z} \times \{3n\}$, and continues along the counterclockwise-most available edge; it is straightforward to see that the continuous path induced by this orientation may be self-touching, but will not self-cross. We orient the rightmost-crossing connecting $I_{i+2}$ to the top starting at $\mathbb{Z} \times \{-3n\}$, with the same counterclockwise-most convention. This defines the $\times$-path $[ab]$ (on the left) and $[cd]$ (on the right). We also explore $|h|$ on the complement of $\Lambda_{3 \rho n, 3n}$. If the exploration fails (i.e. there is no such vertical crossing connecting the appropriate intervals), we explore $|h|$ on all vertices of $\tilde{D}$. We now define $\cF$ to be the $\sigma$-algebra generated by this exploration.

We observe that $\cE^{-2}\cap(\cA^{i-2})^c$ is $\cF$ measurable. To see this, suppose that there is a $\times$-path of $|h| \ge 2$ that connects $[ab]$ to $I_{i-2}$. Then the event $\cA_{i-2}$ has occurred. If no such path exists, then {\em all} the connected components of $|h| \ge 2$ that intersect $I_{i-2}$ must be revealed by the exploration, and all are separated from $[ab]$ by a $|h| \leq 1$ path from $I_{i-1}$ to $\mathbb{Z}\times \{3n\}$. In this scenario, we can determine whether there are paths connecting $I_{i-2}$, $J_{j-2}$, $K_{k-2}$, and $L_{\ell-2}$ with the geometric constraints required by the index triplet $\alpha\beta\gamma$, and thus $\cE^{-2} \cap (\cA^{i-2})^c$ is $\cF$-measurable. Similarly, $\cE^{4} \cap (\cA^{i+2})^c$ is $\cF$-measurable.

\begin{remark}\label{remark:geometry}
We digress a bit with a crucial remark about the geometry of the paths $[ab],[cd]$ and their relationship with \cref{prop:weak,prop:strong}.
When $\cE^{-2}\cap(\cA^{i-2})^c$ and $\cE^{4}\cap(\cA^{i+2})^c$ occur, the geometry of the paths $[ab]$ and $[cd]$ are severely constrained. Let us first discuss the restriction of the geometry to $R_n^-$ and its behavior given the value of $\alpha$ (an analogous statements hold for $R_n^+$ and $\gamma$). If $\alpha = 0$, it is immediate from the definition that {\em every} subpath of $[ab]$ and $[dc]$ that crosses $R_n^-$ must stay inside $\Lambda_{n/2,n} + (2i \eps n, -3n)$ --- that is, the rectangle of height $n$ and width $n/2$, centered around $I_i$. Let us move on to the case $\alpha = +$. Set $\bar{d}$ to be the final intersection of $[dc]$ and $\mathbb{Z} \times \{-n\}$ (note that the path is oriented from top-to-bottom). If $\alpha = +$, the path $[\bar d c]$ must intersect the vertical line containing $b$, and thus also the line containing $b + (\tfrac12,0)$ (since it is part of the rightmost path and there is a path to the right of it which intersects $(2i-11)\ve n$). A similar statement holds for the path $[\bar{a} b]$, defined analogously, and the vertical line containing $c$ when $\alpha = -$. These constraints will be useful later when we apply \cref{prop:weak}, in particular the fact that one of the items $(i)$-$(iii)$ there must occur for $[ab],[cd]$.

For the $R_n^0$ case, we must be more careful in defining the appropriate crossings. Let us consider subpaths of $[ab]$ and $[dc]$ which cross $R_n^0$ from  $\cup_{x=-2}^4J_{j+x} $ to $\cup_{x=-2}^4 K_{k+x}$ remaining inside $\Z \times [-n,n]$. Topological constraints force such subpaths to exist if $\cE^{-2}\cap(\cA^{i-2})^c$ and $\cE^{4}\cap(\cA^{i+2})^c$ occurs (recall $\cE^k$ is a translate of $\cE_{ijkl}$ which forces a subpath to connect $J_{j-2},K_{k-2}$ and $J_{j+4},K_{k+4}$ inside $\Z \times [-n,n]$). Take one such subpath  and we denote the $b'$ and $c'$ to be the intersection of these paths with $\mathbb{Z} \times \{-n\}$. If $\beta =0$, both of these subpath are included in $\Lambda_{n/2,n} +(2j \eps n, -n)$. Indeed, note that on $\cA_{i-2}^c \cap \cA_{i+2}^c$, there exist $|h| \le 1$ crossings on both sides of these subpaths, and if $\beta=0$ these $|h| \le 1$ crossings are `almost vertical' by design which forces these subpaths to stay inside the said thin rectangle. For the same reason, if $\beta = +$, the subpath of $[dc]$ must intersect the vertical line containing $b'$; the analogous statement holds for $\beta = -$, the subpath of $[ab]$, and $c'$. Therefore, for such crossings, one of the items $(i)$-$(iii)$ in \cref{prop:strong} must occur. We emphasise that these topological constraints are forced on any such subpath of $[ab]$ and $[dc]$. Later on we will choose two special such crossings which are `closest' to each other in some sense to define the domain on which we apply \cref{prop:strong}.
\end{remark}

We define $\Omega_0$ to be the domain enclosed by these two paths and the two even $\times$-paths $[bc]$ and $[da]$ between $b$ and $c$ and $d$ and $a$ in $\Z\times\{-3n,-3n-1\}$ and $\Z\times\{3n,3n+1\}$, respectively. We observe that, thanks to the convention chosen above, $(\Omega_0,a,b,c,d)$ is a quad. Setting $B = \partial \Omega_0\cup\Omega_0^c$, we let $\xi$ denote the value of $|h|$ on $B$.


}

Let us come back to the proof of \cref{eq:max_A2}. The conclusion of the above discussion about measurability of the exploration is that it is now sufficient to show that for every  $\xi \in \cE^{-2}\cap\bar{\cE}\cap\bar{\cE}^2\cap\cE^4\cap(\cA^{i-2})^c\cap(\cA^{i+2})^c$, 
\[
\phi_{\tilde{D}}^{B,\xi} [ \cA^i] \geq 1/32.
\] 

 Let $(B,\xi_0)$ be the $|h|$-adapted boundary condition equal to $|h|=2$ on $[ab]\cup[cd]$ and 0 on the rest of the even vertices in $B$.
The comparison between boundary conditions for $|h|$ shows that,
\begin{equation}\label{eq:sym1}
\phi_{\widetilde D}^{B,\xi}[\cA^{i}]\ge \phi_{\widetilde D}^{B,\xi_0}[\cA^{i}]\ge\tfrac14 \phi_{\Omega_0}^{\kappa_0}[\cC_{h\ge2}^\times(\Omega_0,a,b,c,d)],
\end{equation}
where $\kappa_0$ is the boundary condition on $\Omega_0$ equal to $2$ on $[ab]\cup[cd]$ and 0 on $(bc)\cup(da)$. To see the final inequality, we begin by noting that, under the assumed boundary conditions,
 $\cC_{|h|\ge2}^\times(\Omega_0,a,b,c,d)$ implies $\cA^i$. We then observe that
 \[
 \phi_{\widetilde D}^{B,\xi_0} [\cC_{|h|\ge2}^\times(\Omega_0,a,b,c,d)] \geq  \phi_{\widetilde D}^{B,\xi_0} [\cC_{|h|\ge2}^\times(\Omega_0,a,b,c,d) | h|_{[ab]}=2 , h|_{[cd]}=2] \cdot  \phi_{\widetilde D}^{B,\xi_0} [ h|_{[ab]}=2 , h|_{[cd]}=2].
 \]
From the domain Markov property, the first probability above is exactly $\phi_{\Omega_0}^{\kappa_0}[\cC_{h\ge2}^\times(\Omega_0,a,b,c,d)]$. Meanwhile, $\phi_{\widetilde D}^{B,\xi_0} [ h|_{[ab]}=2 , h|_{[cd]}=2]$ is either $1/2$ or $1/4$, depending on whether the left and right boundary are connected by a $|h| \geq 1$ path; regardless, the probability is at least $1/4$. 

%
%
%
%
%

\begin{figure}
\centering
\includegraphics[width = 0.8\textwidth]{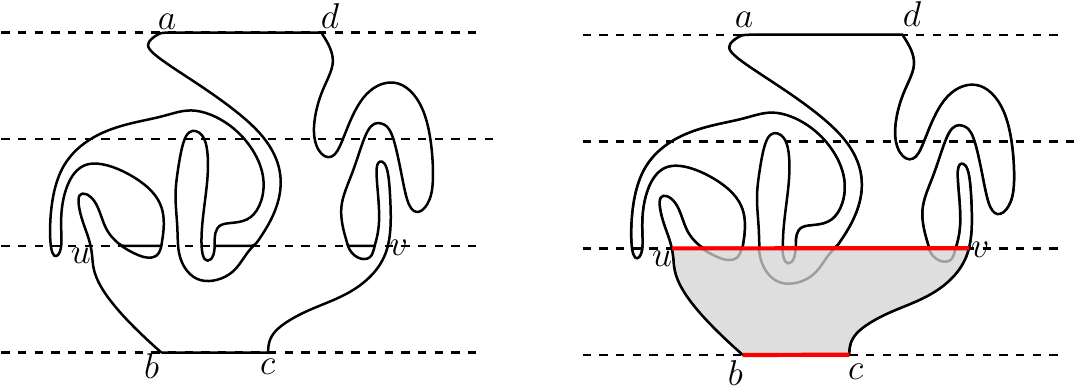}
\caption{The surgery in the definition of $\Omega_-$ (shaded grey). The black paths are even $\times$-paths with value 2 and the red paths are even $\times$-paths with value 0.}
\end{figure}

Overall, we see that, in order to conclude the proof, it is sufficient to show that, for any realization of $\Omega_0$,
\begin{align} \phi_{\Omega_0}^{\kappa_0}[\cC_{h\ge2}^\times(\Omega_0,a,b,c,d)]\ge \tfrac18.\end{align}
{Let $u$ and $v$ be the vertices of $[ab]$ and $[cd]$ which are the last vertices of $\Z \times \{-n\}$ which they intersect respectively (oriented top to bottom). Consequently, $[ub]$ and $[cv]$ are vertical crossings of the strip between height $-3n$ and $-n$ and let $\Omega_-$ be the part of $\Z^2$ enclosed by  $[ub] \cup [bc] \cup [cv]$ and the even $\times$-path $[vu]$ between $u$ and $v$ in $\Z\times\{-n,-n+1\}$ with boundary condition $\kappa_-$ equal to 0 on the bottom and top arcs, and 2 on the rest of the boundary. We now write down a chain of inequalities, which is a common theme in the rest of the proofs. Essentially, the logic is a combination of exploiting measurability of events with respect to absolute value of $h$ and modifying the geometry of domains using FKG for $|h|$ or simply $h$ as appropriate. We first write the chain of inequalities, and provide the explanations for each line below.
\begin{align}
\phi_{\Omega_0}^{\kappa_0}[\cC_{h\ge 1}(\Omega_-,u,b,c,v)]&\ge \phi_{\Omega_0}^{\kappa_0}[\cC_{h\ge 1}(\Omega_-,u,b,c,v)|h_{|\partial\Omega_-\setminus\partial\Omega_0}=0]  \nonumber \\
&=\phi_{\Omega_0}^{\kappa_0}[\cC_{h=2}^*(\Omega_-,u,b,c,v)|h_{|\partial\Omega_-\setminus\partial\Omega_0}=0]\nonumber \\
&\ge\phi_{\Omega_-}^{\kappa_-}[\cC_{h= 2}^*(\Omega_-,u,b,c,v)]\nonumber\\
&=\phi_{\Omega_-}^{\kappa_-}[\cC_{h\ge 1}(\Omega_-,u,b,c,v)]\label{eq:pushing away}.
\end{align}
\begin{itemize}
\item In the first line, we use FKG for $|h|$. Indeed, with the boundary conditions $\kappa_0$, horizontal crossing by a path of $h \ge 1$ is equivalent to crossing by $|h| \ge 1$ (recall that, since the path has regular $\Z^2$ connectivity, its existence is measurable with respect to absolute value). Also notice that $\kappa_0$ is $|h|$-adapted with top and bottom arcs being $B_{\rm sym}$ (with $\kappa \equiv \{0\}$) and the left and right arcs being $B_{\rm pos}$ (with $\kappa \equiv \{2\}$). As per \cref{prop:FKG_modh}, we can introduce new vertices with boundary value 0, thereby decreasing the probability of the crossing, since $|h| \ge 1$ is an increasing event in absolute height value.
\item The equality in the second line simply follows from \cref{lem:duality}.
\item For the third inequality, we note that the map $h \mapsto 2-h$ send a $h=2$ $*$-path to a $h=0$ $*$-path while exchanging the $2$'s and $0$'s on the boundary. Now, we notice that we have a $|h|$-adapted boundary condition and the crossing by a $*$-path of $h=0$ is actually a decreasing function of $|h|$. Hence, we can remove the optimal conditioning 0 on $\partial \Omega_0 \setminus \partial \Omega_- $ and make them `free' ($\kappa \equiv \Z$) thereby decreasing the probability. Globally adding 2 and flipping the sign, we get the required probability in the third line. For later use of a similar step, we will simply say `by FKG for $|h-2|$, we get the third inequality by pushing away the $h=2$ boundary'.
\item The fourth line is \cref{lem:duality}, again.
\end{itemize}
}

Thanks to Remark~\ref{remark:geometry}, we know that $\Omega_-$ satisfies the geometry required in Proposition~\ref{prop:weak} (exactly which of the three conditions occurs depends on the value of $\alpha$, as described in the remark). Thus, the proposition allows us to deduce that, with probability $1/2$, there is a crossing of $h\ge1$ from $[ub]$ to $[cv]$ in $\Omega_-$. 
One can do the same in a domain $\Omega_+$ defined in a similar fashion in the strip $\Z\times[n,3n]$, and FKG for height implies that both crossings occur with probability at least $1/4$.
\bigbreak
We now assume that the event $\cC( \Omega_-, u, b, c, v)$ and the analogous event for the top domain occur in $\Omega_0$. By Lemma~\ref{lem:duality}, this implies that $\Omega_-$ and $\Omega_+$ both contain a $\times$-crossings of $h=1$ from $[ub]$ to $[cv]$. {Further observe that, because of the boundary conditions, one can find such crossings inside $\Omega_0$.}
 Condition on the bottom-most and top-most such $\times$-crossings of $h=1$ and let $\Omega_1$ be the subdomain of $\Omega_0$ enclosed between these paths (we need to further explore inside $\Omega_0$ to find these paths, and orient them appropriately to obtain a quad). We denote by $\kappa_1$ the boundary condition induced by the conditioning, which is 2 on the even vertices of the boundary and 1 on the odd ones. Let $[a'b']$ and $[d'c']$ be subpaths of $[ab]$ and $[dc]$ such  that 
 {
\begin{itemize}[noitemsep,nolistsep]
\item are contained in  $\Z\times[-n,n]$, 
\item are crossing $\Z\times[-n,n]$ from $\cup_{x = -2}^4 J_{j+x}$ to $\cup_{x= -2}^4 K_{k+x}$ from top to bottom,
\item are such that $[a'b']$ is on the left of $[d'c']$ and there is no additional crossing of $\Z\times[-n,n]$ in $[ab]\cup[cd]$ in between.
\end{itemize}
}
\begin{figure}
\centering
\includegraphics[width = \textwidth]{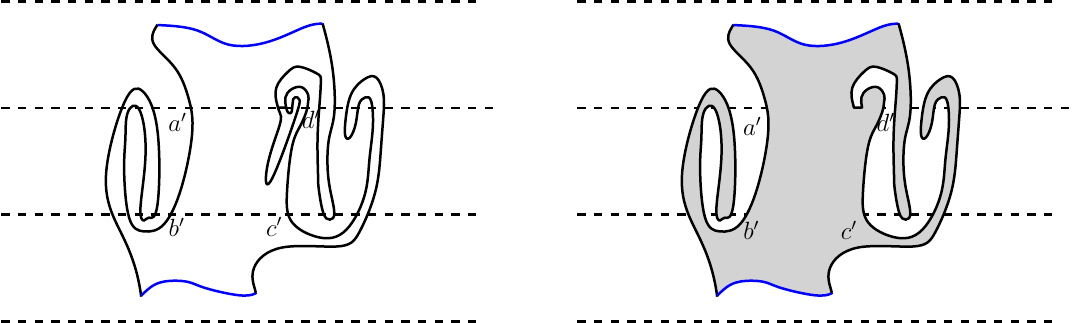}
\caption{The black paths are even $\times$-paths with value 2 and the blue paths are odd $\times$-paths with value 1. Left: The domain $\Omega_1$. Right: The domain $\Omega_2$ in shaded grey obtained after the surgery. }\label{surgery2}
\end{figure}
The existence of such paths is easy to see. Indeed, the first item and third items are obvious and the second item is already explained in \cref{remark:geometry}.  Let $\Omega'_1$ be the vertices of $\Z\times[-n,n]$ between $[a'b']$ and $[c'd']$. Let $\Omega_2$ be the union of the vertices of $\Omega_1$ and $\Omega_1'$, as well as all the vertices surrounded by them (the gray domain in \cref{surgery2}).
In words, $\Omega_2$ corresponds to cutting all the ``tongues'' of $[ab]$ and $[cd]$ entering $\Omega'_1$ by ``pushing them away''. Note that, since we took $[a'b']$ and $[c'd']$ to be successive crossings, none of these tongues crosses $\Omega'_1$.

We now make the observation that in the bottom strip, no part of the boundary with $h=1$ is affected by this procedure. Consequently, we have that $h \geq 2$ on $\partial \Omega_1 \setminus \partial \Omega_2$ (this is crucial for the FKG application later). Indeed, on the one hand any boundary vertex with $h=1$ is connected to $\Z\times\{-3n\}$ or $\Z \times \{3n\}$ by a path staying in $\Omega_0\cap(\Z\times[-3n,-n])$ (or $\Omega_0\cap(\Z\times[n,3n])$).  On the other hand, by definition $\partial \Omega_1 \setminus \partial \Omega_2\subset\Z\times[-n,n]$ and any vertex in $\Omega_2 \setminus \partial \Omega_1'$ is disconnected from $\Z\times\{-3n,3n\}$ by $\partial \Omega_1'$. 
Let $\kappa_2$ be the boundary condition equal to 2 on even vertices of $\partial\Omega_2$, and 1 on odd ones.  

Exactly as for \eqref{eq:pushing away}, FKG for $|h-2|$ enables us to push away the $h=2$ boundary condition to get that 
\begin{align}
\phi_{\Omega_1}^{\kappa_1}[\cH_{h\ge2}^\times(\Omega_1,a,b,c,d)]&=\phi_{\Omega_1}^{\kappa_1}[\cH_{h=2}^\times(\Omega_1,a,b,c,d)]\nonumber\\
&=\phi_{\Omega_2}^{\kappa_2}[\cH_{h=2}^\times(\Omega_1,a,b,c,d)|h_{|\partial\Omega_1\setminus\partial\Omega_2}=2]\nonumber\\
&\ge\phi_{\Omega_2}^{\kappa_2}[\cH_{h=2}^\times(\Omega_2,a,b,c,d)]\nonumber\\
&=\phi_{\Omega_2}^{\kappa_2}[\cH_{h\ge 
2}^\times(\Omega_2,a,b,c,d)]\nonumber\\
& \ge \phi_{\Omega_2}^{\kappa_2}[\cH_{h\ge 
2}^\times(\Omega_2,a',b',c',d')].
\end{align}
Let us elaborate a bit. In the first equality, we used the second item of \cref{lem:duality}, the second equality is simply using the spatial Markov property, the first inequality is FKG for $|h-2|$ and inclusion (a horizontal crossing of $\Omega_2$ guarantees a horizontal crossing of $\Omega_1$ because of the boundary condition). The last equality is again the second item of \cref{lem:duality} and the final inequality is simply inclusion.

Now, let $\Omega_3$ be the mix-domain composed of $\Omega_2$  together with the odd vertices outside $\Z\times[-n,n]$ that are on the exterior boundary of $\Omega_2$ (meaning that they do not belong to the set but are neighbours of a vertex belonging to the set). 
If $\kappa_3$ is the boundary condition on $\Omega_3$ equal to 2 on $[a'b']\cup[c'd']$ and 1 on $(b'c')\cup(d'a')$. We see that
\begin{align}\label{eq:pushing away odd}
\phi_{\Omega_2}^{\kappa_2}[\cH_{h\ge
2}^\times(\Omega_2,a',b',c',d')]&=\phi_{\Omega_3}^{\kappa_3}[\cH_{h\ge
2}^\times(\Omega_2,a',b',c',d')|h_{|\partial\Omega_2\setminus\partial\Omega_3}=2]\nonumber\\
&\ge \phi_{\Omega_3}^{\kappa_3}[\cH_{h\ge2}^\times(\Omega_3,a',b',c',d')].
\end{align}
{The inequality is clear using FKG for $h$, since the boundary condition on $\partial\Omega_2\setminus\partial\Omega_3 =2$ is the highest possible height one can put on these vertices while maintaining the constraint forced by $\kappa_3$. }


As explained in the second paragraph of \cref{remark:geometry}, we now have a quad which is in the first configuration of \cref{prop:strong} if $\beta=0$ (resp.~second if $\beta=+$, third if $\beta=-$). We deduce that 
$$\phi_{\Omega_3}^{\kappa_3}[\cH_{h\ge2}^\times(\Omega_3,a',b',c',d')]\ge\tfrac12,$$ which together with \eqref{eq:pushing away odd}, concludes the proof of \eqref{eq:max_A2}, and consequently that of \cref{thm:RSW}. 
\subsection{Proof of Proposition~\ref{prop:weak}}\label{sec:3.2}

We prove the first two cases of the proposition; the third can be proven analogously to the second.

\paragraph{First case.}

Let $\Lambda_n^{\rm even}$ be the set of vertices inside (or on) the even circuit in $\Lambda_{n+1}\setminus\Lambda_{n-1}$ surrounding the origin.
Consider the boundary condition $\xi$ on $\Lambda_n^{\rm even}$ equal to $2$ on left and right (including vertices on $y=x$), and $0$ on the rest.
By the second item of Lemma~\ref{lem:duality}, and FKG for $|h-2|$ (the reasoning is the same as in \eqref{eq:pushing away}), we find that
\begin{align}
 \phi_{D}^{\kappa}[\cC_{h\ge1}(D,a,b,c,d)]=  \phi_{D}^{\kappa}[\cC_{h=2}^*(D,a,b,c,d)]
  &\ge \phi_{\Lambda_n^{\rm even}}^{\xi}[\cH_{h=2}^*(\Lambda_n^{\rm even})]=\phi_{\Lambda_n^{\rm even}}^{\xi}[\cH_{h\ge1}(\Lambda_n^{\rm even})].\label{eq:pushing away ge1}
\end{align}
Now, using the first item of Lemma~\ref{lem:duality} again, we find that
\begin{align}
\phi_{\Lambda_n^{\rm even}}^{\xi}[\cH_{h\ge1}(\Lambda_n^{\rm even})]&=1-\phi_{\Lambda_n^{\rm even}}^{\xi}[\cV_{h\le0}^\times(\Lambda_n^{\rm even})]\nonumber\\
&\ge 1-\phi_{\Lambda_n^{\rm even}}^{\xi}[\cV_{h\le 1}(\Lambda_n^{\rm even})],\nonumber\\
&\ge 1-\phi_{\Lambda_n^{\rm even}}^{\xi}[\cH_{h\ge1}(\Lambda_n^{\rm even})].\label{eq:sym}
\end{align}
\begin{figure}[h]
\centering
\includegraphics[width = 0.6 \textwidth]{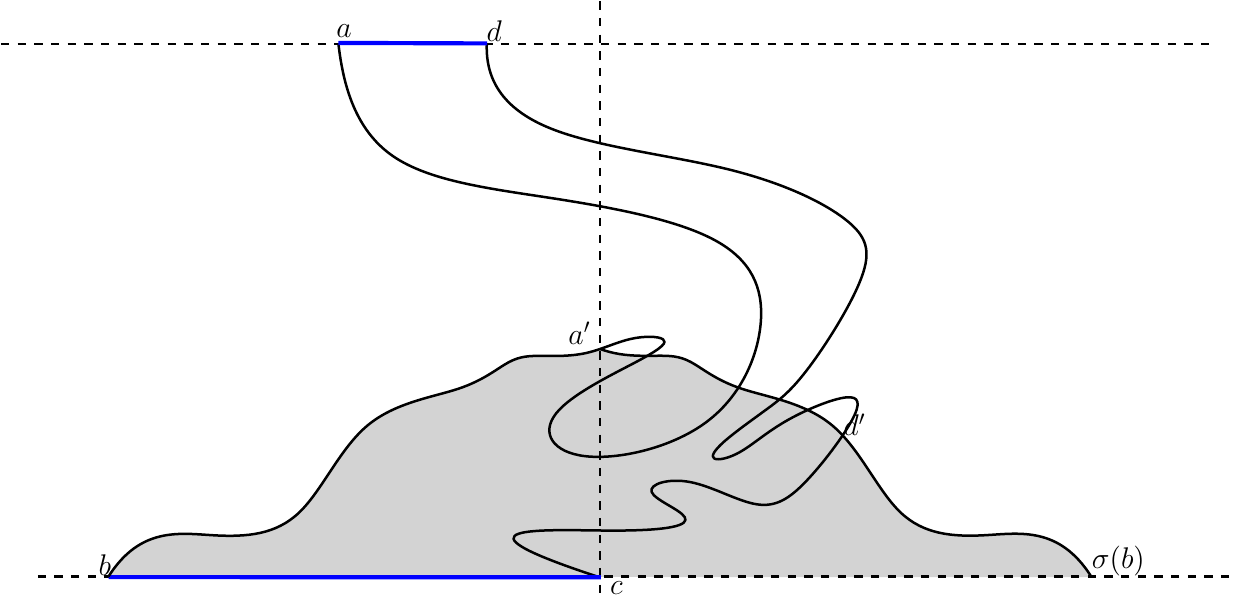}
\caption{The black paths are even paths with value 2 and the blue paths are odd paths with value 1. The symmetric domain $\Omega$ is shaded grey.}
\end{figure}

In order to deduce the last inequality, we used the FKG inequality for $h$ and the symmetry of $\Lambda_n^{\rm even}$ by $\pi/2$ rotation and the fact that the boundary condition $\xi'$ obtained by rotating and applying the transformation $2-h$ is smaller than the boundary condition $\xi$. Overall, \eqref{eq:sym} implies that $\phi_{\Lambda_n^{\rm even}}^{\xi}[\cH_{h\ge1}(\Lambda_n^{\rm even})]\ge\tfrac12$, which concludes the proof.

\paragraph{Second case.} Let $\ell$ be the vertical line passing by $c$. Assume that $[ab]$ and $[cd]$ do not intersect --- otherwise, there is nothing to prove.  
Denote the reflection with respect to the line $\ell$ by $\sigma$, and note that $\sigma$ maps the even lattice to itself.
Let $a'$ be the first intersection of $[ab]$ when going from bottom to top (i.e.~when going from $b$ to $a$) with $\ell$ and let $\Omega$ be the quad enclosed by $[a'b]$, $\sigma([a'b])$ and the $\times$-path of even vertices in $\Z\times\{-n-1,-n\}$ between $b$ and $\sigma(b)$. By definition, $(\Omega,b,c,\sigma(b),a')$ is symmetric under $\sigma$. 

{The idea is to apply the ideas of \eqref{eq:pushing away ge1}, \eqref{eq:sym}, but with the even domain $(\Omega,b,c,\sigma(b),a')$ playing the role of the `square'  domain $\Lambda^{\rm even}_n$ with boundary condition 2 on $[a'b]\cup[c\sigma(b)]$ and $0$ on $(bc)\cup(\sigma(b)a')$. The symmetry  about $\pi/2$-rotation for the `square' domain is replaced by reflection about the line $\ell$ for $(\Omega,b,c,\sigma(b),a')$. Apart from these differences, the proof is analogous to the first case. We provide the details below.}
\medskip

Let $D'$ be the domain bounded by $[a'b]$, $[bc]$, $[cd'],$ where $d'$ the first intersection of $[cd]$ with $\sigma([a'b])$, and $[d'a']$, which is a segment in $\sigma([a'b])$.

{First observe that, because of boundary conditions and inclusion, we have the following inequality
$$
\phi_D^\kappa[\cH_{h\ge 1}(D,a,b,c,d)]\ge \phi_D^\kappa[\cH_{h\ge 1}(D',a',b,c,d')].
$$}
Now applying FKG inequality to $|h|$ (note that the event $\cH_{h\ge 1}(D')=\cH_{|h|\ge 1}(D')$ is increasing in terms of $|h|$), we find that, like in \eqref{eq:pushing away}, 
 \begin{align*}
 \phi_D^\kappa[\cH_{h\ge 1}(D',a',b,c,d')]&\ge\phi_{D}^{\kappa}[\cH_{h\ge1}(D',a',b,c,d')|h_{|\partial D'\setminus\partial D}=0]\\
&= \phi_{D'}^{\kappa'}[\cH_{h\ge 1}(D',a',b,c,d')],\end{align*} 
where $\kappa'$ is the boundary condition equal to $2$ on $[a'b]\cup[cd']$, and 0 on $(bc)\cup(d'a')$. Now, following a reasoning similar to \eqref{eq:pushing away ge1} and then \eqref{eq:sym}, we find that 
\begin{align*}
 \phi_{D'}^{\kappa'}[\cH_{h\ge 1}(D',a',b,c,d')]
 \ge \phi_{\Omega}^{\xi}[\cH_{h=2}^*(\Omega,a',b,c,\sigma(b))]\ge \tfrac12,
\end{align*}
where $\xi$ is the boundary condition equal to 2 on $[a'b]\cup[c\sigma(b)]$ and $0$ on $(bc)\cup(\sigma(b)a')$.

\subsection{Proof of Proposition~\ref{prop:strong}}\label{sec:3.3}
Again, we prove the first two cases of the proposition, as the third is analogous to the second.

Before embarking on the proof of the first case, we recall that the proof of~\cref{prop:weak} used symmetries that mapped even vertices to even vertices, and and symmetric even-domains with boundary condition that are made of 0s and 2s. In this section, we will use symmetries that are mapping even vertices to odd vertices, and symmetric mix-quads with boundary condition that are made of 1s and 2s. Furthermore, unlike the proofs in Section~\ref{sec:3.2}, we are not allowed to `push in' boundary condition larger than or equal to 1 on top and bottom, because $\times$-crossings of $h \ge1$ cannot be transformed into increasing events in $|h|$. We therefore need to symmetrize the domain $D$ by pushing boundary condition 2 away only.

\paragraph{First case.}  Let $\Lambda_n^{\rm mix}$ be the domain enclosed between (see the bottom figure in \cref{fig:new_hard_easy})
\begin{itemize}
\item the even $\times$-paths in $\{n,n+1\}\times\Z$ connecting $(n,n)$ and $(n, -n)$ (call it $[vu]$) and the even times path in $\{-n,-n+1\}\times\Z$ connecting the vertices $(-n+1,-n+1)$ and $(-n+1,n+1)$ (call this path $[sr]$)
\item  The two odd paths obtained by rotating $[uv]$  and $[rs]$ by $\pi/2$ about the point $(\tfrac12,\tfrac12)$.
\end{itemize}
This represents the `mixed square domain', which is necessary to maintain the correct parity in the boundary conditions later in the proof. 
\begin{figure}[h]
\centering
\includegraphics[width = 0.8\textwidth, page =2]{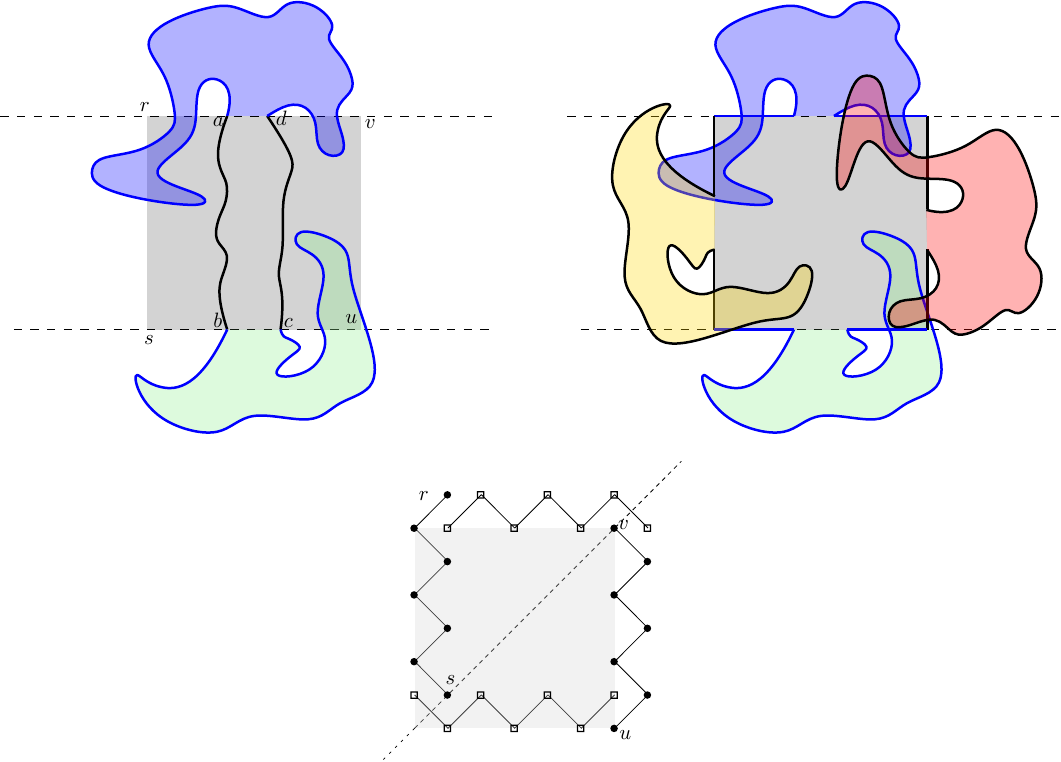}
\caption{Top figures: Blue paths are odd paths with $h=1$ and black paths are even paths with $h=2$. Left: The domain $D$ with the square $\Lambda_n$ in shaded grey. Right: The symmetric domain $\Omega$.  Bottom: The symmetric `mixed square'; the circles are even vertices, and the boxes are odd vertices.}\label{fig:new_hard_easy}
\end{figure}

Let $U_{\rm top}$ be the domain enclosed by the (odd) $\times$-path $[da]$ in the quad $D$ and the (odd) $\times$-path from $a$ to $d$ in $\Z\times\{n,n+1\}$ (the shaded blue domain in \cref{fig:new_hard_easy}). We define $U_{\rm right}$ as the rotation of $U_{\rm top}$ by $-\pi/2$ about the point $(\tfrac12,\tfrac12)$ (shaded red in \cref{fig:new_hard_easy}). Similarly, we define $U_{\rm bottom}$ and $U_{\rm left}$ to be rotation of $U_{\rm bottom}$ by $-\pi/2$ about the point $(\tfrac12,\tfrac12)$  (shaded light green and yellow ochre respectively in \cref{fig:new_hard_easy}). Notice that it is possible to define this domain because of the geometric constraints on the paths $[rs],[uv]$ specified by the first case, namely these paths are contained in the square $\Lambda_{n/2,n}$.

We now introduce  
\[
\Omega = \Lambda_n^{\rm mix} \biguplus U_{\rm top} \biguplus U_{\rm left} \biguplus U_{\rm bottom} \biguplus U_{\rm  right},
\]
where $\biguplus$ denotes the disjoint gluing of the graphs along the top and bottom parts of $\Lambda_n^{\rm mix}$. Note that $\Omega$ is planar but may not be properly embeddable in an isometric fashion in $\R^2$, and that it contains a natural copy of the graph $D$. The graph $\Omega$ also satisfies symmetry with respect to the rotation operation above. Let $\xi$ be the boundary condition on $\Omega$ equal to 2 on even vertices of $\partial\Omega$ and 1 on odd ones. 

 Using the same reasoning as in \eqref{eq:pushing away}, we find that 
 \begin{align}
\phi_{D}^{\kappa}[\cC_{h\ge2}^\times(D,a,b,c,d)]
&\ge\phi_{\Omega}^{\xi}[\cC_{h\ge2}^\times(\Omega,r,s,u,v)].\end{align}
It remains to observe that by the first item of Lemma~\ref{lem:duality}, 
\begin{align}
\phi_{\Omega}^{\xi}[\cC_{h\ge2}^\times(\Omega,r,s,u,v)]&=1-\phi_{\Omega}^{\xi}[\cC_{h\le1}(\Omega,v,r,s,u)]\nonumber\\
&\ge 1-\phi_{\Omega}^{\xi}[\cC_{h\le1}^\times(\Omega,v,r,s,u)].\label{eq:sym2}\end{align}
By symmetry, this implies that the probability of the former is larger than $\tfrac12$. This concludes the proof.

\begin{figure}[h]
\centering
\includegraphics[width  =0.8 \textwidth]{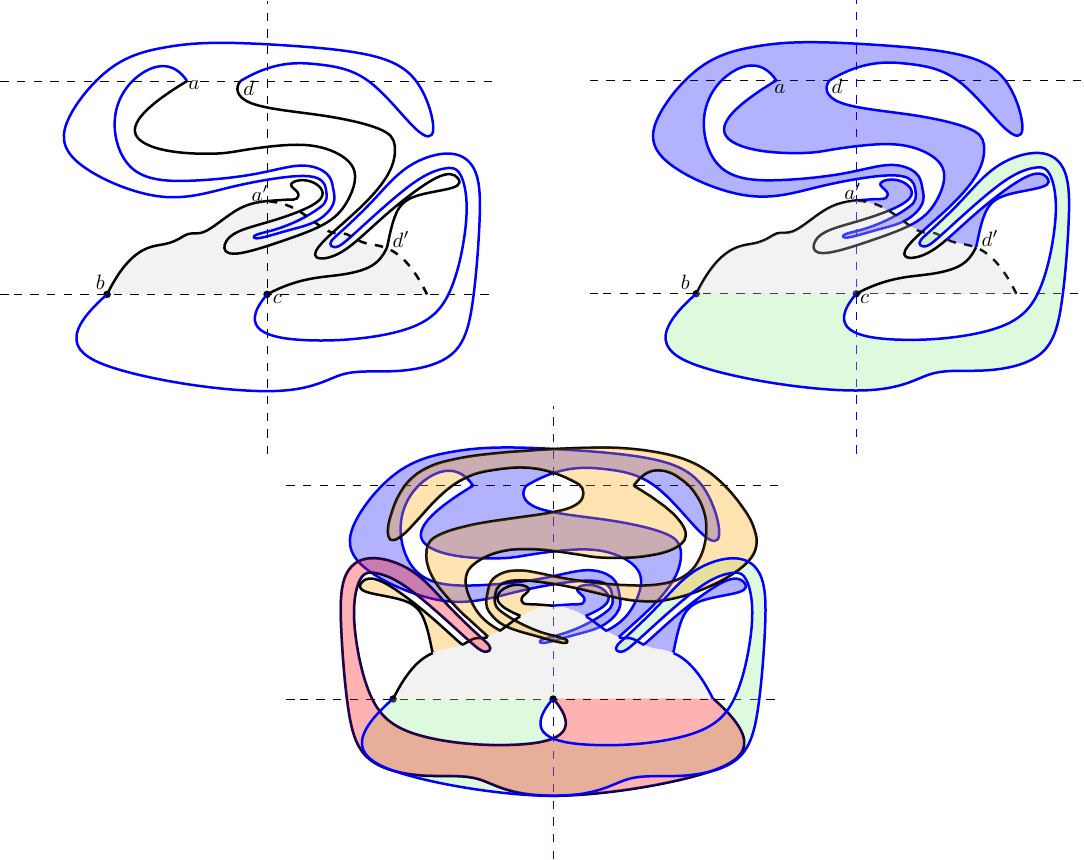}
\caption{Blue curves denote odd $\times$ paths with boundary condition 1 while the black curves are even $\times$ paths boundary condition 2. Top left: The symmetric domain $S$ shaded in grey. Top right: The portion $U_{\rm top-right} $ is shaded blue and $U_{\rm bottom-left}$ is shaded green. Bottom: The final symmetric domain $\Omega$. $U_{\rm top-left}$ is shaded orange and $U_{\rm bottom-right} $ is shaded red.} \label{fig:new_hard_case_surgery}
\end{figure}


 \paragraph{Second case}
Let $\ell'$ be the vertical line passing by $c+(\tfrac12,0)$. We start by defining an equivalent of the symmetric domain introduced in the second case of Proposition~\ref{prop:weak}. Assume that $[ab]$ and $[cd]$ do not intersect otherwise there is nothing to prove.  
Denote the reflection with respect to the vertical line $\ell'$ passing by $c+(\tfrac12,0)$ by $\sigma'$ ($\sigma'$ maps the even lattice to the odd lattice). We now define a domain which will play the role of the domain $\Lambda_n^{\rm mix}$ from the first case. This domain will be symmetric about $\ell'$ so we do not need to shift the reflected paths to maintain the parity constraints. We suggest that the reader refer to \cref{fig:new_hard_case_surgery} while reading the definitions below.

Let $a'$ be the vertex just before the first intersection of $[ab]$ (when going from bottom to top and when seen as a continuous path) with $\ell'$ and let $S$ be the quad enclosed by
\begin{itemize}
\item $[a'b]$,
\item $\sigma'([a'b])$,
\item  the odd $\times$-path of $\Z\times\{-n,-n+1\}$ between the right neighbour of $b$ and the left neighbour of $c$,
\item and the even $\times$-path between $c$ and $\sigma'(b)$.
\end{itemize}
Finally, let $d'$ be the last vertex of $[cd]$ before it exits $S$ for the first time.

Let $U_{\rm top-right}$ be the union of odd domains whose boundaries are defined as follows. Let $s$ be the odd path defined as the union of 
\begin{itemize}
\item the path from a neighbour of $d'$ to neighbour of $d$ using the odd vertices which are neighbours of the vertices of $[dd']$ from the exterior of $D$,
\item the odd path $(da)$,
\item the odd vertices neighbouring from the exterior of $D$ the vertices of $[aa']$.
\end{itemize}

Let $t$ be the odd $\times$-path going along $\sigma([a'b])$ from $d'$ to the neighbour of $a'$ and observe that $t\cap D$ is divided into segments. Now, let $[u_1v_1],\dots, [u_iv_i]$ be the segments of $t$ that can be reached from $a'$ (or equivalently $d'$) while staying in $D\cap S$.
For $1\le j\le i$, consider the domain $U_j$ enclosed by $[u_jv_j]$ and the part of $s$ going from $u_j$ to $v_j$. We then define $U_{\rm top-right}$ to be the union of the $U_j$ for $1\le j\le i$ and $U_{\rm top-left}=\sigma'(U_{\rm top-right})$. 
Similarly, we define $U_{\rm bottom-left}$ and $U_{\rm bottom-right}=\sigma'(U_{\rm bottom-left})$ in a straightforward fashion (in this case the definition is even simpler: there is only one domain since $(bc)$ does not cross the odd $\times$-path of $\Z\times\{-n,-n+1\}$ between the right neighbour of $b$ and the left neighbour of $c$).

Introduce
\[
\Omega = S \biguplus U_{\rm top-left} \biguplus U_{\rm bottom-left} \biguplus U_{\rm bottom-right} \biguplus U_{\rm  top-right}
\]
(the gluing of the different pieces is made along the segments $[u_jv_j]$ defined above for the top pieces, and the odd $\times$-path of $\Z\times\{-n,-n+1\}$ between the right neighbour of $b$ and the left neighbour of $c$ for the bottom parts). 
Let $\xi$ be the boundary condition on $\Omega$ equal to 2 on even vertices of the boundary and 1 on odd ones. 

To complete the proof, we must compare the crossing probabilities in $D'$ and $\Omega$ by pushing away the boundary condition $h=2$, and then applying a symmetry argument. The proof follows the same procedure as the previous case, and we therefore omit the details for the sake of brevity.

\section{Proofs of the theorems}\label{sec:4}

\subsection{Two useful crossing probabilities}\label{sec:4.1}

\begin{proposition}\label{prop:0}
For every $\rho>0$, there exists $c=c(\rho)>0$ such that for every $n$ and every even domain $D$ containing $\Lambda_{\rho n,n}$, 
\begin{equation}
\phi_D^0[\cH_{h\ge0}(\Lambda_{\rho n,n})]\ge c.
\end{equation}
\end{proposition}
\begin{remark}
It is worth mentioning that we do not require $\partial D$ to be far away from $\Lambda_{\rho n,n}$. In fact, the boundary of $\partial D$ may partially coincide with the boundary of $\Lambda_{\rho n, n}$ without raising any issues.
\end{remark}
\begin{proof}
The first item of Lemma~\ref{lem:duality} implies that 
$$p:=\phi_D^0[\cH_{h\ge0}(\Lambda_{\rho n,n})]= 1-\phi_D^0[\cV_{h<0}^\times(\Lambda_{\rho n,n})].$$
{The idea now is the following. 
\begin{itemize}
\item On the one hand, the square root trick implies that several vertical $\times$-crossings of~{$h<0$}, localized to intersect specific locations in the bottom boundary of the rectangle, occur simultaneously 
with positive probability which converges to one as $p$ goes to zero. By sign flip symmetry and the union bound, we may deduce that similar $h>0$ crossings occur in an interlacing fashion with the $h<0$ with positive probability which also goes to one as $p$ goes to zero.
\item On the other hand, the $h<0$ $\times$-crossings can be bridged by $h \le 0$ paths with positive probability (uniformly in $p$), using same arguments as in \cref{sec:3.1}. However, the existence of interlacing $h >0$ $\times$-vertical paths preclude the occurrence of such bridging paths. Combining these estimates lead to a lower bound on $p$.
\end{itemize}
We now provide the details. For the bridging part of the argument, we only need to bridge using $h \le 0$ path, so we only need the arguments of \cref{prop:weak}, and hence we do not require to divide the rectangle into three parts as in \cref{sec:3.1}. To that end recall the definition of $I_k,L_k$ from \cref{sec:3.1}, where we replace $3n$ by $n$ where appropriate. We furthermore recall the definition of $\alpha$ from \cref{sec:3.1}, which specified the geometry of the vertical paths. Now define  $\cE=\cE_{i\ell}^\alpha$ analogously to $\cE_{ijk\ell}^{\alpha\beta\gamma}$, replacing $|h| \geq 2$ by $h<0$ and requiring only one geometric index $\alpha$ (instead of the triplet $\alpha,\beta,\gamma$). We trust the reader to make the appropriate modification required in these definitions. As in \eqref{eq:partition}, the square-root trick implies that there exist $i,\ell,\alpha$ and $C  = C(\rho)$ such that  
$$\phi_D^0[\cE]\ge 1-p^{1/C},$$
For the rest of the proof, call a $\times$-path $\gamma$ {\em achieving} if it guarantees the occurrence of $\cE$. For a subset $I$ of $I_i$, let $\cE(I)$ be the event that there exists an achieving $\times$-path starting from $I$.

We now claim that $I_i$ can be split in eight intervals $I^1,\dots,I^8$ ordered from left to right and intersecting at their extremities (they may have different sizes) such that $$\phi_D^0[\cE(I^k)]\ge1-p^{1/(8C)} \text{ for each $k$}.$$ Indeed, first split the interval $I_i=[ab]$ in two by choosing the left-most $x$ such that the probability that $\phi_D^0[\cE([ax])]\ge\phi_D^0[\cE((xb])]$. Then, the square-root trick implies that both $\phi_D^0[\cE([ax])]$ and $\phi_D^0[\cE([xb])]$ are larger than $1-p^{1/(2C)}$. One can iterate this reasoning to get the claim. 

If $\widetilde\cE(I)$ denotes the image of the event $\cE(I)$ after flipping all the signs, the flip symmetry and the union bound give that 

\begin{equation}
\phi_D^0[\cE(I^1)\cap\widetilde\cE(I^2)\cap\cE(I^3)\cap \widetilde\cE(I^4)\cap \cE(I^5)\cap\widetilde\cE(I^6)\cap\cE(I^7)]\ge 1-7p^{1/(8C)}.\label{eq:lower_big_event}
\end{equation}

We now describe how to use the bridging arguments of \cref{sec:3.1} to obtain an upper bound of the probability of the event above. Let $\cA_i$ be the event that $I_i$ and $I_{i+2}$ are connected by a $h \le 0$ path in $\Lambda_{\rho n, n}$. Now observe that if $\tilde \cE_{i+1} $ occurs, $\cA_i$ cannot occur by duality (\cref{lem:duality}). Thus, if the equality 
\begin{equation}
\max\{\phi_D^0 (\cA_1), \phi_D^0 (\cA_5)\} >\frac13 (\phi_D^0(\cE(I^1) ))^4\ge \frac13 (1-p^{1/(8C)})^4. \label{eq:lower_ca}
\end{equation}
holds, then 
\begin{multline*}
\phi_D^0[\cE(I^1)\cap\widetilde\cE(I^2)\cap\cE(I^3)\cap \widetilde\cE(I^4)\cap \cE(I^5)\cap\widetilde\cE(I^6)\cap\cE(I^7)] \le \min  (\phi_D^0 (\cA_1^c), \phi_D^0(\cA_5^c))\\
\le 1- \frac13 (1-p^{1/(8C)})^4.
\end{multline*}
Combining this upper bound with \eqref{eq:lower_big_event}, we easily obtain a lower bound of $p$ depending only on $C=C(\rho)$ which completes the proof of the proposition.

Let us assume that \eqref{eq:lower_ca} does not hold. In that case
\begin{equation}
\phi_D^0[\cE(I^1)\cap\cA_1^c \cap\cE(I^3)\cap \cE(I^5)\cap \cA_5^c \cap\cE(I^7)] \ge \frac13  (\phi_D^0(\cE(I^1) ))^4 \ge \frac13 (1-p^{1/(8C)})^4\label{eq:complement_a}
\end{equation}

 Let $\bar \cE \supset \cE$ be the event $\cE$ but with no restriction on the geometry or the location where the path hits the top boundary of the rectangle; that is, we only restrict that the starting point of the crossing must be in $I_i$.  

\begin{align}
\phi_D^0[\cE(I^1)\cap\widetilde\cE(I^2)\cap\cE(I^3)& \cap \widetilde\cE(I^4)\cap \cE(I^5) \cap\widetilde\cE(I^6)\cap\cE(I^7)] \\ \le &\phi_D^0[\cE(I^1)\cap\cA_1^c  \cap \bar \cE(I^3)\cap \cA_3^c \cap \bar \cE(I^5)\cap\cA_5^c \cap\cE(I^7)]\label{eq:reduction_A}\\
\le & 1-\phi_D^0[\cE(I^1)\cap\cA_1^c  \cap \bar \cE(I^3)\cap \cA_3 \cap \bar \cE(I^5)\cap\cA_5^c \cap\cE(I^7)] \label{eq:01} \\  =: & \, 1 - \phi_D^0[E \cap \cA_3] 
\end{align}
where we emphasize that $\cA_3^c$ changes to $\cA_3$ in the final inequality above. Furthermore, we have
\begin{equation}
\phi_D^0[E \cap \cA_3] = \phi_D^0[\cA_3 | E]  \cdot \phi_D^0[E] \ge \phi_D^0[\cA_3 | E] \cdot  \frac13 (1-p^{1/(8C)})^4\label{eq:02'}
\end{equation}
where in the last line we used \eqref{eq:complement_a}. Now we claim that
\begin{equation}
\phi_D^0[\cA_3 |E] \ge \frac18.\label{eq:03}
\end{equation}
Indeed, we are exactly in the setup of \eqref{eq:max_A2}, except we are in a strictly easier case. Firstly mapping $h \mapsto 1-h$, the above event is the same as connecting two vertical $\times$-paths of $h \ge 2 $ by a path of $h \ge 1$ with boundary condition $h=1$ on the boundary of $D$. By FKG for $h$, this event has strictly larger probability than the same with boundary condition equal to 0. We now proceed in the same way as in the proof of \cref{thm:RSW} in \cref{sec:3.1}. First, we explore to reduce to the case in \eqref{eq:sym1} with $h \ge 2$ replaced by $h \ge 1$ thereby losing a factor of $\frac14$ (notice that connecting two vertical $h \ge 2$ $\times$-paths by a $|h| \ge 1$ path is the same event as connecting them by $h \ge 1$ path, hence using FKG for absolute value \cref{prop:FKG_modh} we can push the 0 boundary in). Then we can apply \cref{prop:weak} to show that this event has probability at least $\frac12$. Thus overall, we get a lower bound of $\frac18$. \footnote{In \eqref{eq:max_A2}, we needed to apply \cref{prop:weak} twice and \cref{prop:strong} once which gave a factor of $\frac1{32}$ instead of $\frac18$.} Inserting \eqref{eq:03} into $\eqref{eq:02'}$, and further inserting that in \eqref{eq:01} and applying \eqref{eq:lower_big_event}, we see that
$$
 1-\frac1{24} (1-p^{1/(8C)})^4 \ge 1-7p^{1/(8C)}.
$$ 
This easily yields a lower bound on $p$ depending only on $C = C(\rho)$, thereby completing the proof. 
}

\end{proof}
{The second estimate we wish to obtain could be thought of as an enhancement of \cref{prop:0} when $D$ is a rectangle. We wish to show that in a rectangle (with any aspect ratio) with boundary condition 0 on the left, right and top boundary and any integer $g \ge 0$ on the bottom boundary, the probability of obtaining a $\times$-crossing of $h \le 0$ in an arbitrarily thin rectangle close to the bottom boundary is positive (depending only on the aspect ratios and $g$, but not on the scale). We first define the rectangle and the boundary conditions in a proper way as there are (mild) technical issues with the parity of the paths and the compatibility of boundary condition for an arbitrary $g$ (for example, we cannot require $g>2$ to be a $\times$-neighbour of 0).}

When $n$ is even, consider the approximation $R_{n,m}^g$ of $[-n,n]\times[0,m]$ obtained by taking what is inside 
\begin{itemize}[noitemsep,nolistsep]
\item the even $\times$-path going from $(n,0)$ to $(-n,0)$ following $\{n,n+1\}\times\Z$, then $\Z\times\{m,m+1\}$ and finally $\{-n-1,-n\}\times\Z$,
\item if $g$ is even (resp.~odd), the even (resp.~odd) $\times$-path from $(n,0)$ to $(-n,0)$ in $\Z\times\{-1,0\}$.
\end{itemize}
A similar approximation can easily be defined for $n$ odd by replacing $n$ above by $2\lfloor n/2\rfloor$.
Also, let $0/g$ be the boundary condition equal to $g$ on the bottom of $R_{n,m}^g$, $0$ on the left, right and top boundary, except at the bottom-left and bottom-right corners where the boundary condition is interpolating between $0$ and $g$ in the shortest way. Since the superscript will always be obvious from context (for instance because it is the only one compatible with the boundary conditions), we will write $R_{n,m}$ instead of $R_{n,m}^g$. {We will abuse terminology here and call $R_{n,m}$ a domain in what follows.}

\begin{proposition}\label{prop:mixed}
For any $g \in \N$, $H >  \delta > 0$, there exits $c=c(H,\delta,g)>0$ such that for all $n \ge 1$,
\begin{align*}
\phi^{0/g}_{R_{n,Hn}}[\cH_{h= 0}^\times(R_{n,\delta n})]\ge\phi^{0/g}_{R_{n,Hn}}[\cH_{h\le 0}(R_{n,\delta n})] \geq c.
\end{align*}

\end{proposition}

\begin{proof} The first inequality clearly follows by inclusion.  {The second inequality will follow from two steps. In step 1, we assume the result is true for $g=1$, and then induct on $g$. In step 2, we prove the base case $g=1.$

\paragraph{Step 1: Induction step assuming result for $g=1$.} This step is simply an application of FKG for $|h| $ iteratively.
Indeed, assume that we already proved the existence of $c(H,\delta,1)>0$. Then we immediately see that we can find a $h\le g-1$ crossing within $R_{n,\delta n}$ with positive probability $c(H,\delta,1)>0$. Indeed, notice that if $g$ is odd, we can consider the boundary condition $\phi^{g-1/g}$ which is an overall increase of the boundary values which implies the lower bound of the required probability by FKG  for $h$ and the existence of $c(H,\delta,1)$. If $g$ is even, the required application of FKG is not much more complicated: we can first increase the 0 boundary values to $g-2$, and then using domain Markov property put values $g-1$ on the exterior boundary of the top, bottom and left, and then drop the boundary condition with value $g-2$. All these operations decrease the probability of $h \le g-1$ crossing. We leave the details of this to the reader for brevity.}

Now condition on the bottom-most crossing of $h\le g-1$ (we explore from the bottom in a Markovian way as usual). We wish to find a crossing of $h \le g-2$ above this crossing. Noting that the event ``there exists a crossing of $R_{n,2\delta n}\setminus R_{n,\delta n}$ of $h\le g-2$'' is increasing in terms of $|h-g+1|$, we can use FKG for $|h - g+1|$ to push the boundary conditions to get the translate of the $0/(g-1)$ boundary condition on the translate by $(0,\delta n)$ of $R_{n,(H-\delta)n}$, so that the conditional probability of finding a crossing of $h \le g-2$ inside $R_{n,2\delta n}\setminus R_{n,\delta n}$ (and therefore inside $R_{n,2\delta n}$) is bounded from below by $c(H-\delta,\delta,1)$. Iterating, we see that we can find a crossing of $h \le 0$ inside $R_{n,g\delta n}$ with probability at least $c(H,g\delta,g):=\prod_kc(H-k\delta,\delta,1)>0$.
\paragraph{Step 2: proof for $g=1$.}
We now focus on proving the existence of $c(H,\delta,1)>0$. Assume without loss of generality that $Hn$ is divisible by 4. Let $S_n$ be the infinite strip bounded by the even vertices of $\Z \times \{Hn+1, Hn+2\}$ and the odd vertices of $\Z \times \{0,-1\}$ and let  $0/1$ be the boundary condition equal to 0 on the top and 1 on the bottom. The existence of the measure $\phi^{0/1}_{S_n}$ is a straightforward exercise as the domain is essentially one dimensional and the homomorphism model enjoys a version of the finite energy property. Let $R'_n:=[-n,n]\times[\tfrac14Hn,\tfrac34Hn]$.

We claim that there exists a constant $c=c(H)>0$ such that for all $n \ge 1$,
\begin{equation}\label{eq:hor crossing}
\phi^{0/1}_{S_n}[\cH_{h=0}^\times(R_{n,3Hn/4})]\ge \phi^{0/1}_{S_n}[\cH_{h\le0}^\times(R'_n)] \geq c.
\end{equation}
Indeed, the first inequality follows from the inclusion of events (induced by boundary conditions). For the second, assume that it does not hold with $c=1/2$. Then, the first item of Lemma~\ref{lem:duality} and the symmetry of the measure imply that
$$
\phi^{0/1}_{S_n}[\cV_{h\le0}^\times(R'_n)]\ge\phi^{0/1}_{S_n}[\cV_{h\le0}(R'_n)]=\phi^{0/1}_{S_n}[\cV_{h\ge1}(R'_n)]=1- \phi^{0/1}_{S_n}[\cH_{h\le0}^\times(R'_n)]> \tfrac12.$$
{We now repeat the proof of Theorem~\ref{thm:RSW} to show that, given the lower bound on the probability of vertical crossings, we can produce a lower bound on the probability of horizontal crossings. First, we use the $+/-$ and translation symmetries and the FKG inequality in $h$ to observe that 
\begin{equation}
2 \phi^{0/1}_{S_n}[\cH_{h\le0}^\times(R'_n)] = 2 \phi^{2/1}_{S_n}[\cH_{h\ge 2}^\times(R'_n)] \geq \phi^{2/1}_{S_n}[\cH_{|h|\ge 2}^\times(R'_n)]. 
\end{equation}
From here, we assume that there exist there vertical $\times$-crossing of $R_n'$ with $h \leq 0$, localized to begin at distinct intervals of length $2 \varepsilon n$ on the bottom, and explore the outermost realizations of such crossings (as we did in Section~\ref{sec:3.1}). From here, the remaining argument follows, {\em mutatis mutandis}.}

%

The FKG inequality for $|h|$ enables us to bring boundary conditions in to find that 
\begin{equation}
\phi^{0/1}_{R_{n,Hn}}[\cH_{h=0}^\times(R_{n,3Hn/4})]\ge \phi^{0/1}_{S_n}[\cH_{h=0}^\times(R_{n,3Hn/4})]\ge c.\label{eq:push1}
\end{equation}
Now, condition on the top-most horizontal $\times$-crossing of $h=0$ in $R_{n,3Hn/4}$ {(as usual by exploring in a Markovian way starting from the top boundary).} Applying spatial Markov property and FKG of $|h|$, the probability of seeing a $\times$-crossing of $h=0$ in  $R_{n,(3/4)^2Hn}$ is larger than $c(\tfrac34 H)$ { since, in this case, we can `straighten' the top boundary by pushing it inside using FKG for $|h|$}. We iterate this step to 
find that for all $n \ge 1$,
\[
\phi^{0/1}_{R_{n,Hn}}[\cH_{h=0}^\times(R_{n,\delta n/2})]\ge \prod_{k\le \log_{4/3}(2H/\delta)}c((\tfrac34)^kH)>0.
\]
To conclude, it remains to create a $h\le0$ crossing in $R_{n,\delta n}$ from the weaker $h \le 0$ $\times$-crossing analyzed above. Explore from the bottom to find the lowest such $\times$-crossing of $h=0$ and call the domain above it $D$. Conditioned on this lowest crossing, the boundary condition on $D$ is 0 everywhere. We therefore can apply Proposition~\ref{prop:0} to show that the probability of a crossing of $h\le0$ in $R_{n,\delta n}\setminus R_{n,\delta n /2}$ is bounded from below by $c'>0$ {(this is where we used that the boundary of $D$ in \cref{prop:0} can touch the boundary of the rectangle)}, thus proving that 
\[
\phi^{0/1}_{R_{n,Hn}}[\cH_{h\le 0}(R_{n,\delta n})]\ge c'\phi^{0/1}_{R_{n,Hn}}[\cH_{h=0}^\times(R_{n,\delta n/2})].
\]
Combined with the previous displayed equation, this concludes the proof.
\end{proof}

\subsection{The renormalization proposition}\label{sec:4.2}

As described in the introduction, we wish to follow the renormalisation argument from \cite{DST} to complete the argument. Unfortunately, a new difficulty appears in our setting: one could imagine that the existence of a long $\times$-crossing of $h \geq 2$  inside a box forces the height function to be much larger than $2$ everywhere inside. In practice, it manifests in the fact that to apply \cref{prop:mixed}, we need a bound on the boundary values which is not {\em a priori} clear.

To deal with this issue, we distinguish two cases. If the probability of a crossing of $h \geq 2$ is similar to the one of a $\times$-crossing of $2 \leq h \leq g$ for some $g$, which is the expected behaviour, we can apply the original argument of \cite{DST} with appropriate modifications. If not, then the cost of a $\times$-crossing of  $h \geq 2$ is similar to the cost of a crossing of $h \geq g$. In this case we obtain $(g-2)/2$ crossings ``for free'', an event whose probability can be easily bounded. 

Define $\cA_n$ to be the event that there exists a $\times$-loop of $h \geq 2$ in the annulus $A_n:=\Lambda_{2n} \setminus \Lambda_n$ and let $a_n := \phi_{\Lambda_{5n}}^0[\cA_n]$. We also let $\cA_n(x)$ denote the event $\cA_n$ shifted by $(x, 0)$. Finally, we introduce $\Lambda_n(x)$ and $A_n(x)$ for the box $\Lambda_n$ and the annulus $A_n$ shifted by $(x,0)$.

\begin{remark}\label{rmk:upper}
By Proposition~\ref{prop:0} and duality \cref{lem:duality}, we have that $a_n\le 1-c$ for some constant $c>0$ independent of $n$.
\end{remark}

\begin{proposition}\label{prop:renormalization}
There exists  a constant $C>0$ such that for all $n \ge 1$,
\begin{equation}
a_{10n} \leq C a_n^2. \label{eq:renormalization}
\end{equation}
\end{proposition}

We start with a lemma. Let $\cE^\times_{h\ge k}(n)$ be the event that there exists a $\times$-cluster of $h\ge k$ of diameter at least $n$.

\begin{lemma}\label{lem:case2}
There exists $\rho>0$ such that the following holds.
For any $r>10$, there exists 
$C = C(r)>0$ such that for any $k$ and $n$, \begin{equation}\label{eq:case2claim}
\phi_{\Lambda_{r n}}^0[\cE^\times_{h\ge k}(n)] 
\leq  (Ca_{n})^{k/(8\rho)}.
\end{equation}
\end{lemma}

\begin{figure}[h]
\centering
\includegraphics[scale = 0.5]{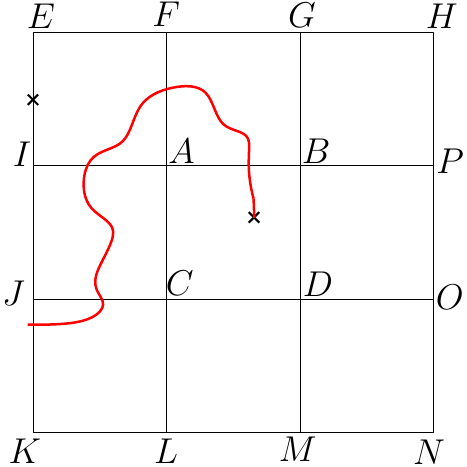}
\caption{The red curve has diameter at least $n$ and hence has to exit the square EKNH. Without loss of generality let the curve exit the square ABCD for the last time through AB before it exits EKNH. After this assume it exits the square AFGB through AF or GB, since otherwise we are done and without loss of generality assume it is AF. Then we can assume it exits IBGE through IA since otherwise we are also done. After this, the curve must necessarily include an easy crossing of either IJDB or EJCF before exiting EKNH.}\label{fig:large_easy}
\end{figure}

\begin{proof}
In any configuration of $\cE^\times_{h\ge k}(n)$, there must exist $k/2$ nested $\times$-loops with increasing heights $2i\le k$.
Let $D_{2i}$ be the interior of the outer-most $\times$-loops of $h=2i$ in $\Lambda_{rn}$ and let $\kappa_i$ denote the boundary condition equal to $2i$ on $D_{2i}$. On the event described above, the  domains $D_{2i}$ exist for every $2i\le k$ and if the diameter of the largest connected component of these domains is denoted by $d_{2i}$, we find that 
\[
\phi_{\Lambda_{rn}}^0[\cE^\times_{h\ge k}(n)] 
 \leq \phi_{\Lambda_{rn}}^0\big[ \mathbbm1_{d_2\geq n} \phi^{\kappa_2}_{D_2}[ \mathbbm 1_{d_4 \geq n} \phi^{\kappa_4}_{D_4}[\cdots] ] \big].
\]
Using the symmetry in $D_{2i}$ with respect to $2i$ and the FKG for $|h-2i|$, we may upper bound each expectation by $2\phi^0_{\Lambda_{rn}}[d_2 \geq n ]$ (the factor $2$ arises from switching between $|h|$ and $h$), giving overall
\begin{equation}\label{eq:hh1}
\phi_{\Lambda_{r n}}^0[\cE^\times_{h\ge k}(n)] 
\le (2\phi_{\Lambda_{rn}}^0[d_2 \geq n ])^{k/2}.
\end{equation}
 Consider the set $T$ of translates of rectangles included in $\Lambda_{rn}$ of sizes $n \times n/2$ and $n/2 \times n$ by  vertices in $\tfrac n2\Z^2$. A topological argument (see Figure~\ref{fig:large_easy}) easily implies that if $d_2\ge n$, there exists a rectangle in $T$ that is crossed in the `easy' direction, meaning vertically if it has size $n\times n/2$ or horizontally if it has size $n/2 \times n$. For this reason, in order to bound $\phi_{\Lambda_{rn}}^0[d_2 \geq n ]$, it suffices to consider a rectangle $R$ in $T$, which we assume without loss of generality has size $n\times n/2$, and to prove that there exists $C_0>0$ such that 
\begin{equation}\label{eq:hh2}\phi_{\Lambda_{r n}}^0[\cV_{h=2}^\times(R)]\le C_0 a_n^{1/(4 \rho)}.\end{equation}
Let $\cA$ be the event that there exists a $\times$-circuit of $h\le 0$ surrounding $R$ in the $n$ neighbourhood of $R$ (if $R$ intersects the boundary, the $\times$-circuit can use the boundary of $\Lambda_{rn}$, which has value 0). We have that 
\begin{equation}
\phi_{\Lambda_{r n}}^0[\cV_{h=2}^\times(R)]\leq 
\frac{ \phi_{\Lambda_{rn}}^0 [\cV_{h\ge2}^\times(R) | \cA]}{ \phi_{\Lambda_{rn}}^0[\cA |\cV_{h=2}^\times(R)]}\le \frac{ \phi_{\Lambda_{2 n}}^0 [\cV_{h\ge2}^\times(\Lambda_{n, n/2} )]}{ \phi_{\Lambda_{rn}}^0[\cA |\cV_{h=2}^\times(R)]}\le C_1 \phi_{\Lambda_{2n}}^0 [\cV_{h\ge2}^\times(\Lambda_{n,n/2} )]. \label{eq:delta_shortening}
\end{equation}
Indeed, the first inequality follows from inclusion. The second holds because the $\times$-loop of $h \le 0$  induced by $\cA$ can be replaced by a $\times$-loop of $h=0$ by FKG for $h$ and then pushed away using FKG for $|h|$ (we already presented several arguments like that and omit the details).
In the third inequality, we lower bound the probability of the denominator as follows. Condition on $|h-2|$ in the even vertices $R^{\rm even}$ in $R$. Any realization of this conditioning is a measure of the form $\phi_{\Lambda_{rn}\setminus R^{\rm even}}^\kappa$ with $\kappa$ being  $|h|$-adapted (the intervals are containing one value on $\partial\Lambda_{rn}$ and two symmetric values in $R^{\rm even}$). Since the single-valued vertices all receive value $h=0$, we may use the comparison between boundary conditions with $h=2$ on $\partial R^{\rm even}$ to bound the conditional probability from below by the $\phi_{\Lambda_{rn}\setminus R^{\rm even}}^0$-probability in $\Lambda_{rn}\setminus R^{\rm even}$ with boundary condition equal to 2 on $\partial R^{\rm even}$ and $0$ in $\partial\Lambda_{rn}$. Using \cref{prop:mixed}, we have a positive probability that $\cA$ occurs, hence the third inequality (we can bring in the 0 boundary using FKG for $|h|$ to apply \cref{prop:mixed} four times, followed by FKG inequality).
 
Now, Theorem~\ref{thm:RSW} implies that
\begin{equation}\label{eq:hhh1}
\phi_{\Lambda_{5n, 2n}}^0 [\cH_{h\ge2}^\times(\Lambda_{2n, n/2})] 
\ge c_0\phi_{\Lambda_{2n}}^0 [\cV_{h\ge2}^\times(\Lambda_{n, n/2} )]^\rho.
\end{equation}
Finally, FKG for $h$  implies that 
\begin{equation}\label{eq:hhh2}
a_{n}\ge \phi_{\Lambda_{5n,2n}}^0 [\cH_{h\ge2}^\times(\Lambda_{2n, n/2} )]^{4}.
\end{equation}
Equations  \eqref{eq:delta_shortening} and \eqref{eq:hhh1} and \eqref{eq:hhh2} can be combined to deduce \eqref{eq:hh2}; this, in turn, can be combined with \eqref{eq:hh1} to imply \eqref{eq:case2claim} and complete the proof.\end{proof}

\begin{figure}[h]
\centering
\includegraphics[width=0.9\textwidth]{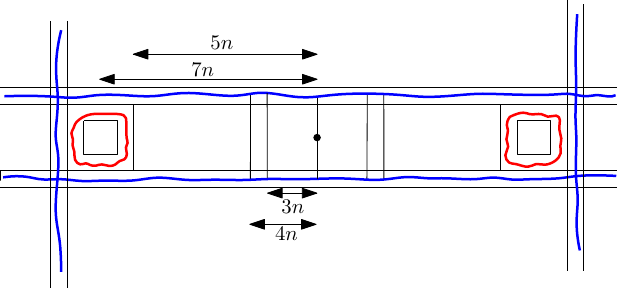}
\caption{The red paths are even $\times $-loops of $h \ge 2$ coming from the event $\cA_n(-7n) \cap \cA_n(7n)$. The blue paths are $\times$-paths of $h \le 0$ coming from the events \eqref{eq:A1}--\eqref{eq:A4} (i.e.~the event $\cC$).}\label{fig:renorm}
\end{figure}
\begin{proof}[Proof of Proposition~\ref{prop:renormalization}]
We start by claiming that there exists $c_1>0$ such that for all $n \ge 1$,
\begin{equation}\label{eq:emain}
a_{10n} \leq c_1 \phi_{\Lambda_{50n}}^0 [ \cA_n(-7n) \cap \cA_n(7n) ]
\end{equation}
since $\cA_{10n}$ implies that there exists a $\times$-loop of $h=2$ in  $\Lambda_{50n} \setminus \Lambda_{10n}$  and hence we can apply Proposition~\ref{prop:0}.

Define events $\cB_n(x)$ similarly to $\cA_n(x)$, where we add the restriction that the $\times$-loops must satisfy $2 \leq h \leq k_0:=2^{\ell_0}$, where $k_0>16\rho$ with $\rho$ provided by Lemma~\ref{lem:case2}. With this choice of $k_0$, Lemma~\ref{lem:case2} gives that  
\[
\phi_{\Lambda_{50n}}^0 [\cA_n(\pm7n) \setminus \cB_n(\pm7n)]\le \phi_{\Lambda_{50n}}^0[\cE_{h\ge k_0}^\times(n)]\le C_1a_n^2.
\]
In order to conclude the proof, we now need to prove that 
\begin{equation}\label{eq:est1}
 \phi_{\Lambda_{50n}}^0[ \cB_n(-7n) \cap \cB_n(7n) ]\le C_2a_n^2.
\end{equation}
Suppose we are on the event $\cB_n(-7n) \cap \cB_n(7n)$ and let $\ell_{\pm}$ be the two innermost $\times$-loops with values in $2 \le h \le k_0$ in $\Lambda_n(\pm7 n)$.

Consider the event $\cC$  (see \cref{fig:renorm}) which is the intersection of the following four events
\begin{align}
\cH_{h\le 0}^\times([-50 n, 50 n]\times[2n, 3n]), \label{eq:A1} \\
\cH_{h\le 0}^\times([-50 n, 50 n]\times[-3n, -2n]),\label{eq:A2}\\
\cV_{h\le0}^\times([-10n,-9 n]\times[-50n, 50n]),\label{eq:A3}\\
\cV_{h\le 0}^\times([9n,10 n]\times[-50n, 50n]). \label{eq:A4}
\end{align}
Conditionally on $\ell_+$ and $\ell_-$ (which involves information on $h$ inside the two loops only), we claim that the $\phi_{\Lambda_{50n}}^0$-probability of $\cC$ is at least $c_3>0$. Let us provide a lower bound on the conditional probability of the first event, since the bound for the others is similar and that one can use FKG to deduce a bound on $\phi_{\Lambda_{50n}}^0[\cC]$. 
Since the values of $h$ on the loops $\ell_+$ and $\ell_-$ are between $2$ and $k_0$ {we can first assume that the value is uniformly $k_0$ on $\ell+ $ and $\ell_-$ by FKG for $h$ since this lowers the probability of the $h \le 0$ crossing by FKG for $h$ (this is where the upper bound of $k_0$ is crucially used.)} Next, the FKG inequality for $|h-k_0|$ enables us to bound the probability of the first event in $\cC$ from below if we assign (the translate of) boundary condition $0/k_0$ on the rectangle $[-50 n, 50 n]\times[2n, 50n]$ as in \cref{prop:mixed} {(we interpolate the values between $k_0$ and 0 near the corners of the rectangle as in \cref{prop:mixed}, and it is easy to see that this is the worst possible boundary value for the required event in terms of FKG for absolute value)}. In other words, it is enough to prove the lower bound for the same event but in the domain $[-50 n, 50 n]\times[2n, 50n]$ with $0/k_0$ boundary condition, which is exactly what is given by \cref{prop:mixed}. 

\begin{remark}Note that this step crucially relies on the fact that the values on $\ell_+$ and $\ell_-$ are bounded between $2$ and $k_0$ since applying FKG for $|h-k_0|$  requires all the boundary values to have the same sign in \cref{prop:FKG_modh}. 
 \end{remark}
Overall, the  argument in the previous paragraph gives us the existence of $c_4>0$ such that for all $n$,
\begin{equation}
\phi_{\Lambda_{50n}}^0 [\cC| \cB_n(-7n) \cap \cB_n(7n) ] \ge c_4.\label{eq:surround}
\end{equation}

\bigbreak
On $\cC \cap \cB_n(-7n) \cap \cB_n(7n)$, let $\Omega$ be the connected component of the origin inside the outermost realisations of the crossings in \eqref{eq:A1}, \eqref{eq:A2}, \eqref{eq:A3}, \eqref{eq:A4} minus the loops $\ell_+$ and $\ell_-$ (see \cref{fig:renorm}) {(we do a standard Markovian exploration from outside inwards to achieve this)}. As in \cite{DST}, we want to separate $\ell_+$ and $\ell_-$ with an $h \leq 0$ $\times$-path. However, since the values on $\ell_-$ and $\ell_+$ can be as high as $k_0$, we need several steps to find this path. We do so iteratively, each time dividing the value of the separating path by a factor of 2. 

We now provide the details. Let $R_-$, $R_0$ and $R_+$ be the subsets of $\Omega$ made of vertices with first coordinates in $[-4n,-3n]$, $[-3n,3n]$, and $[3n,4n]$ respectively. We write $\cV_{h\le 2^{k}}^\times(R_\#)$ for the existence of  a vertical $\times$-crossing of this quad between the bottom and top boundaries of $\partial\Omega$. Let
$$\cD(k):=\cC\cap \cV_{h\le 2^{k}}^\times(R_-)\cap\cV_{h\le 2^{k}}^\times(R_+),$$
and on this event, set $\Omega_k$ to be the part of $\Omega$ between the left-most vertical $\times$-crossing of $h\le 2^k$ of $R_-$ and the right-most vertical $\times$-crossing of $R_+$. We also use the conventions $\cD(\ell_0)=\cC$ and $\Omega_{\ell_0}:=\Omega$. 

We wish to show iteratively that there exist $c_{\ell_0},\dots,c_1>0$ such that for every $1\le k< \ell_0$,
\begin{equation}\label{eq:Dk}\phi_{\Lambda_{50n}}^0[\cD(k)|\cC\cap \cB_n(-7n) \cap \cB_n(7n)]\ge c_k.\end{equation}
Fix $k\ge1$ and assume that the previous result was obtained {for every $\ell_0 \ge k'>k$}. The boundary condition induced on $\Omega_k$ is such that
$$\cV_{h\le 2^{k}}^\times(R_0)=\cV_{|h-2^{k+1}|\ge 2^{k}}^\times(R_0)$$
since the sign of the path must be the same as the boundary by the third item of Lemma~\ref{lem:duality}. Therefore, we can put boundary conditions $2^{k+1}$ on the left and right sides of $R_0$ using FKG for $|h|$. Rewriting this event as $|h| \le 2^{k}$ using Lemma~\ref{lem:duality} again, we can push out the zeros on $\partial\Omega$ to the top and bottom boundaries of $R_0$ ({again, it is easy to check that if $n\ge 2^{k+1}$, it is possible to design boundary conditions that interpolate between $2^{k+1}$ and $0$ in a symmetric fashion in the corners but we voluntarily suppress this issue for the sake of clarity).} By duality (like in the argument for the first case of Proposition~\ref{prop:weak}), we deduce that 
$$\phi_{\Lambda_{50n}}^0[\cV_{h\le 2^{k}}^\times(R_0)|\cD(k+1) \cap \cB_n(-7n) \cap \cB_n(7n)]\ge \tfrac12,$$ 
which together with the induction hypothesis gives
$$\phi_{\Lambda_{50n}}^0[\cV_{h\le 2^{k}}^\times(R_0)|\cC \cap \cB_n(-7n) \cap \cB_n(7n)]\ge \tfrac12c(k+1).$$
On this event, let $\Omega_+$ be the subregion of $\Omega$ on the right of the left-most vertical $\times$-crossings of $R_0$ of $h\le 2^{k}$. Conditionally on $\Omega_+$, we can make the event $\cV_{h\le 2^{k}}^\times(R_+)$ less probable by putting boundary conditions $2^{k}$ to the left, top and bottom sides of $[-3n,4n]\times[-3n,3n]$, and $k_0$ to the right (which is again made compatible near the corners and the parity chosen appropriately) by using 
\begin{itemize}[noitemsep,nolistsep]
\item  FKG for $|h-k_0|$ ($\cV_{h\le 2^{k}}^\times(R_+)=\cV_{|h-k_0|\ge k_0-2^k}^\times(R_+)$ is increasing in $|h-k_0|$ for the boundary conditions on $\Omega_+$) to put $h=k_0$ boundary conditions on the right of $\Omega_+':=\Omega_+\cap(R_0\cup R_+)$, 
\item Comparison between boundary conditions for $h$ ($\cV_{h\le 2^{k}}^\times(R_+)$ is decreasing in $h$) to put $h=2^k$ on the rest of $\partial\Omega_+'$,\item FKG for $|h-2^k|$ ($\cV_{h\le 2^{k}}^\times(R_+)=\cV_{h= 2^{k}}^\times(R_+)$ is decreasing for $|h-2^k|$ for the boundary conditions) to push the boundary conditions to the top, left, and bottom of the rectangle $[-3n,4n]  \times [-3n,3n]$.\end{itemize}
We can apply \cref{prop:mixed} to get that 
\begin{equation}\label{eq:o1}\phi_{\Lambda_{50n}}^0[\cV_{h\le 2^k}^\times(R_+)|\cV_{h\le 2^{k}}^\times(R_0)\cap\cC \cap \cB_n(-7n) \cap \cB_n(7n)]\ge c_1.\end{equation}
Similarly, one can condition on the right of the right-most crossing of $\cV_{h\le 2^{k}}^\times(R_0)$ to get
\begin{equation}\label{eq:02}\phi_{\Lambda_{50n}}^0[\cV_{h\le 2^k}^\times(R_-)|\cV_{h\le 2^k}^\times(R_+)\cap\cV_{h\le 2^{k}}^\times(R_0)\cap\cC \cap \cB_n(-7n) \cap \cB_n(7n)]\ge c_1.\end{equation}
 Forgetting the occurrence of $\cV_{h\le 2^{k}}^\times(R_0)$, we deduce \eqref{eq:Dk} for $c(k):=\tfrac12c_1^2c(k+1)>0$.
 \begin{figure}[h]
\centering
\includegraphics[width = \textwidth]{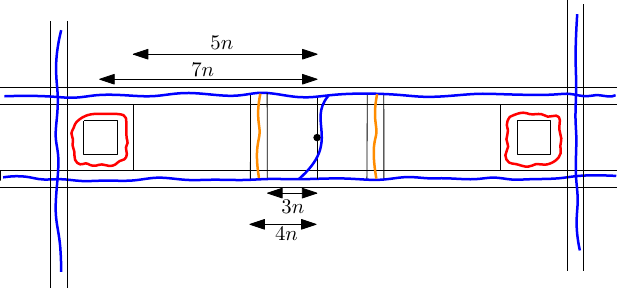}
\caption{The event $\cD(k)$ involves finding the orange $\times$-paths depicted above taking value at most $2^k$. Given $\cD(1)$, we wish to find a  blue $\times$-path between them taking value at most 0 using the bridging \cref{prop:weak}.}\label{F:renorm1}
\end{figure}
\bigbreak
Now that we have the existence of $c(1)>0$, suppose we are on $\cD(1)\cap \cB_n(-7n) \cap \cB_n(7n)$. The first case of Proposition~\ref{prop:weak} in $\Omega_1$ implies that with probability $1/2$, one can construct vertical $\times$-crossings of $h \le 0$  in $R:=R_-\cup R_0\cup R_+$.
Forgetting about the occurrence of $\cD(1)$ and conditioning on the left-most $\times$-crossing of $h\le0$ in $R$, we can again construct a domain $\Omega^+$ and this time deduce the existence of a $\times$-crossing of $h\le 0$ in $[4n,5n]\times[-3n,3n]$ via a reasoning similar to the one leading to \eqref{eq:o1}. Then, one conditions on the right-most $\times$-crossing and deduces a similar claim for $[-5n,-4n]\times[-3n,3n]$. Overall, if $\cE$ is the event that $A_{2n}(-7n)$ and $A_{2n}(7n)$ contain $\times$-circuits of $h\le 0$, the previous reasoning together with \eqref{eq:Dk} gives that
\begin{align}\label{eq:est3}
\phi_{\Lambda_{50n}}^0[\cC\cap\cB_n(-7n) \cap \cB_n(7n)] & \le C_2\phi_{\Lambda_{50n}}^0 [\cE\cap\cA_n(-7n) \cap \cA_n(7n)].
\end{align}
As in \cite{DST}, conditioned on the outer-most circuits of $h\le0$ in $\Lambda_{5n}(-7n)$ and $\Lambda_{5n}(7n)$, the FKG for $|h|$ implies that $\cA_n(-7n)$ and $\cA_n(7n)$ are decoupled events and the probability of $\cA_n(-7n)$ and $\cA_n(7n)$ are each bounded by $\phi_{\Lambda_{5n}}^0[\cA_n]$, so that  
\[
\phi_{\Lambda_{50n}}^0 [\cE\cap\cB_n(-7n) \cap \cB_n(7n)]  \le Ca_n^2.
\]
Combining this inequality with \eqref{eq:surround} and \eqref{eq:est3}, we obtain \eqref{eq:est1}, a fact which conclude the proof.
\end{proof}

\subsection{Proofs of the theorems}\label{sec:4.3}

\begin{proof}[Proof of Theorem~\ref{thm:dichotomy}]
Assume that there exists $m_0$ such that $a_{m_0}  < (2C)^{-1}$ with $C$ being the constant from the renormalisation equation (\cref{prop:renormalization}). Iterating \eqref{eq:renormalization}, there exist $c_1,C_1>0$ such that for every $r\in\bbZ_+$,
\begin{equation}\label{eq:po2}
a_{10^rm_0} \le C_1\exp(-c_12^r) .
\end{equation}
Using RSW and FKG (similarly to the proof of \cref{lem:case2}), there exist $C_0,\rho_0>0$ such that  for all $m \ge 1$,
\begin{equation}\phi_{\Lambda_{5m}}^0[\cV_{|h|\ge1}(\Lambda_{2m,m/2})]\le C_0\phi_{\Lambda_{10m}}^0[\exists\text{ circuit of $|h|\ge1$ surrounding $\Lambda_m$}]^{\rho_0}.\label{eq:po}\end{equation}
 Now, consider $m'$ to be the smallest integer of the form $20^rm_0$ which is larger than $20m$. Using the FKG inequality for $|h|$, one may combine such circuits in $\Lambda_{20m}$ to get the existence of $C_1,\rho_1$ such that
 $$\phi_{\Lambda_{20m}}^0[\exists\text{ circuit of $|h|\ge1$ surrounding $\Lambda_m$}]\le C_1\phi_{\Lambda_{5m'}}^0[\exists\text{ circuit of $|h|\ge1$ surrounding $\Lambda_{2m'}$}]^{\rho_1}.$$
Conditioning on the exterior-most $\times$-circuit of 1 and using FKG for $|h-1|$ and \eqref{eq:po}, we deduce that there exist $C_3,\rho_3>0$ such that 
$$
\phi_{\Lambda_{5m}}^0[\cV_{|h|\ge1}(\Lambda_{2m,m/2})] \le C_3a_{m'}^{\rho_3}.
$$
Together with \eqref{eq:po2}, we find that there exist $C_2,c_2>0$ such that for every $m$,
$$\phi_{\Lambda_{5m}}^0[\cV_{|h|\ge1}(\Lambda_{2m,m/2})]\le C_2\exp(-m^{c_2}).$$
Now, the first item of Lemma~\ref{lem:duality} can be trivially adapted to state that  there is a crossing of $|h|\ge1$ from $\partial\Lambda_m$ to $\partial\Lambda_{2m}$ if and only if there is no $\times$-loop of $h=0$ in $A_m$ surrounding 0, so that
 the $\phi_{\Lambda_{5m}}^0$-probability of this event is bounded by $4\phi_{\Lambda_{5m}}^0[\cV_{|h|\ge1}(\Lambda_{2m,m/2})]\le 4C_2\exp(-m^{c_2})$.

The event that $|h (0)| \ge 2k$ implies that there is no $\times$-loop of $0$ (in fact no 0 at all) in $A_k$ surrounding 0.  As a consequence, there exists an integer $r\ge 0$ such that there is no $\times$-loop of $h=0$ in $A_{2^rk}$ surrounding 0 but there is one in $A_{2^{r+1}k}$. Conditioning on the exterior-most such loop and using the FKG inequality (to push the zero in) and the estimate above, we deduce that 
$$\phi_{\Lambda_n}^0[|h (0)| \ge 4k]\le \sum_{r\ge 0} 4C_2\exp(-(2^rk)^{c_2})\le C_3\exp(-k^{c_3}).$$

We now assume that $a_n\ge (2C)^{-1}$ for every $n$. Note that the first inequality follows trivially from the second one applied to $k+2$ so we only focus on the second inequality. Fix $\ep,\rho>0$ and $k$.

First,
 observe that using loops in successive annuli, there exists a constant $c_0=c_0(k)>0$ such that for all $n$, 
 $$\phi_{\Lambda_{\ep n}}^0[\exists\text{ $\times$-loop in $A_{\ep2^{-k-1} n}$ of $h\ge k$}]\ge c_0.$$
Define the annulus
$A:=\Lambda_{(\rho+\ep)n,n}\setminus \Lambda_{(\rho+\ep/2)n,n/2}$. 
The FKG inequality and the  concatenations of small $\times$-loops with value $h\ge k$ give the existence of $c_1=c_1(k,\ep,\rho)>0$ such that 
$$\phi_D^0[\exists\text{ $\times$-loop in $A$ of $h\ge k$}]\ge c_1.$$
Remark~\ref{rmk:upper} and 
Lemma~\ref{lem:case2} enable us to fix $k_0$ sufficiently large that
$$\phi_D^0[\cE^\times_{h\ge k_0}(\ep n)]\le \tfrac12c_1.$$
Altogether, we conclude that 
$$\phi_D^0[\exists\text{ $\times$-loop in $A$ of $h\in[k,k_0)$}]\ge \tfrac12c_1.$$
Condition on the exterior-most such $\times$-loop and let $\Omega$ be the domain inside. The boundary conditions induced by the conditioning are between $k$ and $k_0$.  Now, combining small $\times$-loops of $h\le k$, we may construct a crossing of $\Lambda_{\rho n,n/2}$ with probability $c_2=c_2(\ep,\rho,k,k_0)>0$. If this happens, we automatically obtain a $\times$-crossing of $R$ of $h=k$. We deduce that this occurs with probability $\tfrac12c_1c_2$, which is independent of $n$ as desired. 
\end{proof}


We now prove logarithmic bounds for the variance of the height function. We first prove them in a box $\Lambda_n^{\rm even}$ introduced in Section~\ref{sec:3.2} (a similar statement can easily be obtained in a generic domain).

\begin{proposition}
There exist $c,C>0$ such that for every $n$,
$$c\log n\le \phi_{\Lambda_n^{\rm even}}^0[h_0^2]\le C\log n.$$
\end{proposition}
\begin{proof}
Let $v_n = \phi_{\Lambda_n^{\rm even}}[h_0^2]$. We start with the lower bound. Let  $\cG_{n}$ be the event that there is a $|h|=2$ $\times$-loop inside $A_n$ which by \cref{cor:main} has $\phi_{\Lambda_{2n}^{\rm even}}^0$-probability at least $c$ independent of $n$. On $\cG_n$, call the vertices lying inside the outermost $|h| =2$ $\times$-loop $\Omega_n$. The bound $v_n \ge c \log n$ for some $c>0$ follows by iterating the inequality
\begin{align*}
v_{2n } &= \phi_{\Lambda_{2n}^{\rm even}}^0[h_0^2 1_{\cG_{n}}] + \phi_{\Lambda_{2n}^{\rm even}}^0[h_0^2 1_{\cG^c_{n}}]\\
&=\phi_{\Lambda_{2n}^{\rm even}}^0 \big[\phi_{\Omega_n}^0[(h_0+\xi)^2]1_{\cG_n}\big] +\phi_{\Lambda_{2n}^{\rm even}}^0\big[\phi_{\Lambda_{2n}^{\rm even}}^0[h_0^2|h_{|\partial\Lambda_{n}^{\rm even}}]1_{\cG^c_n}\big]\\
&=\phi_{\Lambda_{2n}^{\rm even}}^0 \big[(\phi_{\Omega_n}^0[h_0^2]+4)1_{\cG_n}\big] +\phi_{\Lambda_{2n}^{\rm even}}^0\big[\phi_{\Lambda_{2n}^{\rm even}}^0[h_0^2|h_{|\partial\Lambda_{n}^{\rm even}}]1_{\cG^c_n}\big]\\
&\ge (\phi_{\Lambda_{n}^{\rm even}}^0 [h_0^2]+4) \phi_{\Lambda_{2n}^{\rm even}}^0 [\cG_n] + \phi_{\Lambda_n^{\rm even}}[h_0^2]\phi_{\Lambda_{2n}^{\rm even}}^0 [\cG^c_n]\\
&= v_n + 4 \phi_{\Lambda_{2n}^{\rm even}}^0 [\cG_n]\ge v_n+4c.
\end{align*}
where $\xi$ is a random variable taking values $\pm 2$ with equal probability independent of everything else. The justification of this sequence of inequalities is the following. To see the second equality, note that on the event $\cG_n$ we can explore  $|h|$ until we discover $\Omega_n$. The third one follows from the spatial Markov property, the independence of $h$ and $\xi$, and the fact that $\phi^0_{\Omega_n}[h_0]=0$. The inequality follows from the comparison between boundary conditions and the FKG inequality for $|h|$.

Let us now turn to the upper bound. One can implement a proof which is quite similar to the lower bound here, but we choose a different road which extends trivially to the torus case. Consider $\ell_k$ to be the outer-most $\times$-loop of $h\ge2k$ surrounding the origin, if it exists. Also, for each $i\le \log_2 n$ (here we forget the rounding since it does not impact the rest of the proof), let $\mathbf N_i$ be the number of indexes $k$ such that the maximal distance between a vertex in $\ell_k$ and the origin is between $2^i$ and $2^{i+1}$. Observe that 
\begin{align}\label{eq:kh}
\phi_{\Lambda_n}^0[h_0\ge 2N]\le \sum_{N_1+\dots+N_{\log_2 n}=N}\phi_{\Lambda_n}^0[\mathbf N_i=N_i,\forall i\le \log_2 n].
\end{align}
We claim that for every $\ep>0$, there exists $C_0>0$ such that 
\begin{align}\label{eq:good bound}\phi_{\Lambda_n}^0[\mathbf N_i=N_i|\mathbf N_1=N_1,\dots,\mathbf N_{i-1}=N_{i-1}]\le C_0\ep^{N_i}.
\end{align}
Plugging this estimate into \eqref{eq:kh} and using that $\binom ab\le (ea/b)^b$ implies that  \begin{align*}
\phi_{\Lambda_n}^0[h_0\ge 2N]\le  (1+\tfrac {\log n}N)^N C_0^{\log_2 n}(e\ep)^{N}.
\end{align*}
Since we may choose $\ep$ as small as we wish, this quantity decays exponentially fast in $N\ge C_1\log n$, thus concluding the proof. We therefore turn to the proof of \eqref{eq:good bound}. 
\bigbreak
Fix $r>0$. Let $\Omega_k$ be the domain enclosed by $\ell_k$ and set $k_i:=N_1+\dots +N_i$. 
Pave the annulus $A_{2^i}$ by balls of size $2^{i-r}$ centred at $x_0,\dots,x_{R}$. Let $M_k$ be the number of such balls that are intersecting $\Omega_k$. We claim that there exists a constant $c_0>0$ such that for every $k\in(k_{i-1},k_i)$,
\begin{equation}\label{eq:cl}\phi_{\Lambda_n}^0[M_{k+1}=M_k>0|\ell_1,\dots,\ell_{k}]\le (1-c_0)^r\qquad\text{ a.s.}.\end{equation}
Indeed, the conditional measure is, up to a sign, equal to $\phi_{\Omega_k}^{2k}$. Now, since $M_k>0$ and $\ell_k$ intersects the annulus, one may choose $x_k$ such that the ball of radius $2^{i-r}$ around it intersects $\ell_k$. Note that $M_{k+1}=M_k>0$ imposes the occurrence, for every $1\le s\le r$, of the event  $\cE_s$ that there exists a $\times$-crossing of $h=2k+2$ in the annulus $A_s$  around $x_k$ of inner and outer radii $2^{i-s-1}$ and $2^{i-s}$. Using the FKG for $|h-2k-2|$, we therefore deduce that
\begin{equation}
\phi_{\Lambda_n}^0[M_{k+1}=M_k>0|\ell_1,\dots,\ell_{k}]\le\phi_{\Omega_k}^{2k}\big[\bigcap_{s=1}^{r}\cE_s\big]\le \prod_{s=1}^r\phi_{A_s\cap\Omega_k}^{\kappa_s}[\cE_s],
\end{equation}
where $\kappa_s$ is the boundary conditions equal to $h=2k+2$ on the inner and outer boundaries of $A_s$, and $h=2k$ on the rest of the boundary. Since the existence of the $\times$-crossing of $h=2k+2$ from inside to outside is the complement  under these boundary conditions of the existence of a $*$-path of $2k$ from $\ell_k$ to itself, we may use the FKG inequality for $|h-2k|$ and the shifting of the height-function down by $2k$, to bound the probability of the event $\cE_s$ by the event that there exists a $*$-circuit of 0 surrounding the origin in an annulus with boundary conditions 2. This probability is bounded by $1-c_0$ using Corollary~\ref{cor:main},  and we therefore obtain \eqref{eq:cl}. 

Now, if one finds $N_i$ loops with radius between $2^i$ and $2^{i+1}$, there must be at least $N_i-C_2$ indexes $k\in(k_{i-1},k_i)$ for which $M_{k+1}=M_k>0$, where $C_2$ is a function of $r$ only. We deduce that
$$\phi_{\Lambda_n}^0[\mathbf N_i=N_i|\mathbf N_1=N_1,\dots,\mathbf N_{i-1}=N_{i-1}]\le N_i^{C_2}e^{-r(N_i-C_2)}$$
which implies \eqref{eq:good bound} with $\ep$ and 
a constant $C_0=C_0(\ep)>0$ provided that we select $r$ large enough.
\end{proof}

We conclude this article with the proof of Theorem~\ref{thm:torus}.

\begin{proof}[Proof of \cref{thm:torus}] 
Fix a representative of the equivalence class of each homomorphism by setting $h(x)=0$. Using FKG for $|h|$, we deduce that 
$$\phi_{\mathbb T_n}[(h(y)-h(x))^2]=\phi_{\mathbb T_n}^{\{x\},0}[h(y)^2]\ge \phi_{\Lambda_{|x-y|}(y)}^0[h(y)^2]\ge c\log |x-y|.$$
The upper bound can be deduced by an argument similar to the one developed in the last proof (defining the circuits starting from 0 in an outward direction; any non-contractible loops which separate $0$ from $y$ are also included in this list). 

\end{proof}

\bibliographystyle{abbrv}
\def\cprime{$'$}

\end{document}